\documentclass[11pt]{scrartcl}

\usepackage{amsmath,amssymb,amsthm}

\usepackage[english]{babel}
\usepackage[utf8]{inputenc}

\usepackage{csquotes}
\usepackage{tikz}
\usepackage[a4paper,headheight=14pt]{geometry}
\usepackage{enumitem}
\usetikzlibrary{matrix,arrows.meta,quotes,graphs,calc,cd}
\usepackage{csquotes}
\usepackage[backend=biber, style=alphabetic]{biblatex}
\usepackage{hyperref}
\usepackage{microtype}

\makeatletter\def\blx@nowarnpolyglossia{}\makeatother 

\newtheorem{lem}{Lemma}[section]
\newtheorem{prop}[lem]{Proposition}
\newtheorem{thm}[lem]{Theorem}

\newtheorem{corl}[lem]{Corollary}
\theoremstyle{definition}
\newtheorem{dfn}[lem]{Definition}
\theoremstyle{remark}
\newtheorem{rem}[lem]{Remark}

\newenvironment{psmallmatrix}{\left(\begin{smallmatrix}}{\end{smallmatrix}\right)}
\newcommand{\C}{{\mathbb{C}}} 
\newcommand{\Z}{{\mathbb{Z}}}
\newcommand{\R}{{\mathbb{R}}}
\newcommand{\Q}{{\mathbb{Q}}}
\newcommand{\N}{{\mathbb{N}}}

\newcommand{\ev}{\mathrm{ev}}
\newcommand{\llbracket}{[\kern -0.25em[}
\newcommand{\rrbracket}{]\kern -0.25em]}
\newcommand{\lllbracket}{[\kern -0.25em[\kern -0.25em[}
\newcommand{\rrrbracket}{]\kern -0.25em]\kern -0.25em]}
\newcommand{\bbrack}[1]{\llbracket#1\rrbracket}
\newcommand{\bbbrack}[1]{\lllbracket#1\rrrbracket}

\DeclareMathOperator{\Fr}{Fr}

\DeclareMathOperator{\id}{id}

\DeclareMathOperator{\Ad}{Ad}

\DeclareMathOperator{\coker}{coker}
\DeclareMathOperator{\ind}{ind}
\DeclareMathOperator{\asind}{asind}

\DeclareMathOperator{\ch}{ch}
\DeclareMathOperator{\vol}{vol}

\title{Asymptotically flat Fredholm bundles and assembly}
\author{Benedikt Hunger}

\begin{document}

\maketitle

\begin{abstract}
  Almost flat finitely generated projective Hilbert C*-module bundles were successfully used by Hanke and Schick to prove special
  cases of the Strong Novikov Conjecture. Dadarlat later showed that it is possible to calculate the index of a K-homology class
  $\eta\in K_*(M)$ twisted with an almost flat bundle in terms of the image of $\eta$ under Lafforgue's assembly map and the
  almost representation associated to the bundle.
  Mishchenko used flat infinite-dimensional bundles equipped with a Fredholm operator in order to prove
  special cases of the Novikov higher signature conjecture.

  We show how to generalize Dadarlat's theorem to the case of an
  infinite-dimensional bundle equipped with a continuous family of Fredholm operators on the fibers. Along the way, we show that
  special cases of the Strong Novikov Conjecture can be proven if there exist sufficiently many almost flat
  bundles with Fredholm operator.

  To this end, we introduce the concept of an asymptotically flat Fredholm bundle and its
  associated asymptotic Fredholm representation, and prove an index theorem which relates the index of the asymptotic Fredholm
  bundle with the so-called asymptotic index of the associated asymptotic Fredholm representation.
\end{abstract}

\section{Introduction}

Since Novikov made his famous higher signature conjecture \cite{novikov-algebraic-construction-ii}, special cases of the
signature were successfully proven using very different techniques. Already Lusztig \cite{lusztig-novikov-higher-signature}
realized that one can use index theory and families of flat vector bundles to prove Novikov's conjecture for free abelian
groups. Later, Mishchenko \cite{mishchenko-infinite-dimensional-representations-higher-signatures} used families of
infinite-dimensional flat bundles, equipped with Fredholm operators on the fibers, to prove that Novikov's conjecture holds for
nonpositively curved manifolds. Kasparov \cite{kasparov-k-theory-group-C*-algebras-higher-signatures-conspectus} showed that the
higher signature conjecture follows from what was later called the \emph{Strong Novikov Conjecture}, namely the conjecture that
the \emph{analytic assembly map} $\mu_{BG}\colon K_0(BG)\to K_0(C^*G)$ is rationally injective, where $C^*G$ is any C*-algebra
completion of the complex group algebra $\C G$. This analytic assembly map is defined in terms of a certain flat Hilbert
C*-module bundle, the Mishchenko bundle. Connes, Gromov, and Moscovici \cite{connes-gromov-moscovici-conjecture-de-novikov}
realized that it is sufficient to consider not only flat bundles, but bundles which have a very small curvature in order to
prove Novikov's higher signature conjecture.

Hanke and Schick \cite{hanke-schick-enlargeability-and-index-theory-I,hanke-schick-enlargeability-and-index-theory-infinite,
hanke-schick-novikov-low-degree-cohomology,hanke-positive-scalar-curvature} proved that also the Strong Novikov Conjecture holds
if there is a sufficient supply of almost flat finitely generated projective Hilbert C*-module bundles.

Now let us consider the following situation: Let $p\colon\bar M\to M$ be a covering space, and suppose that there exists a
sequence of compactly supported Hilbert $B_n$-module bundles $\bar E_n\to\bar M$ with compatible connections,\footnote{In
particular, we assume that there is a trivialization at infinity such that parallel transport is trivial with respect to this
trivialization.} such that the curvatures of $\bar E_n$ tend to zero as $n$ goes to infinity, and such that every $\bar E_n$
detects the lift $p^!\ch(\eta)$ of the Chern character of a K-homology class $\eta\in K_*(M)$. In the case where all of the
C*-algebras $B_n$ are equal to $\C$, Hanke and Schick \cite{hanke-schick-enlargeability-and-index-theory-infinite} developed a
method to construct finitely generated projective Hilbert C*-module bundles $E_n\to M$ with compatible connection which detect
$\ch(\eta)$ and whose curvatures tend to zero. Note that this is easy if $\bar M\to M$ is a finite cover: then one can simply
take as fiber over $x\in M$ the direct sum of the fibers of the pre-images $\bar x\in p^{-1}(x)$. The construction is, however,
not trivial at all if $\bar M\to M$ is an infinite cover.

Hanke's and Schick's construction does not work if the C*-algebras $B_n$ are arbitrary C*-algebras. Indeed, their construction
makes essential use of the algebra of trace-class operators on Hilbert spaces, i.~e. on Hilbert $\C$-modules, and for
$B_n\neq\C$ there is no appropriate replacement for this algebra. However, in this case one
can still consider the bundles $E_n$ which have as fibers the direct sums of the fibers of the pre-images. The fibers of these
bundles will then be infinite-dimensional over the C*-algebra $B_n$. Since $\bar E$ is assumed to be compactly supported, there
is a trivialization at infinity which induces a map from $E_n$ to a trivial bundle which is unitary modulo compact operators. In
other words, we obtain a graded Hilbert $B_n$-module bundle $E_n'$ together with an odd self-adjoint operator
$F_n\colon E_n'\to E_n'$ whose square equals the identity modulo a fiberwise compact operator. Furthermore, $F_n$ commutes with
parallel transport in $E_n'$ modulo compact operators because we assumed that parallel transport in $\bar E_n$ is trivial at
infinity. Such a pair $(E_n',F_n)$ will be called a Fredholm bundle, and the family $(E_n',F_n)_{n\in\N}$ will be called an
\emph{asymptotically flat Fredholm bundle} because the curvatures of the bundles $E_n'$ tend to zero as $n$ goes to infinity.
Now one can consider the generalized
Fredholm index $\ind F_n\in K_0(C(X)\otimes B_n)$, and observe that $\langle\eta,\ind F_n\rangle\neq 0$ for all $n$. We will
prove in Theorem~\ref{thm:application SNC} that in this situation again $\mu_{B\pi_1(M;*)}\psi_*\eta\neq 0\in
K_*(C^*\pi_1(M;*))$, which generalizes the results of Hanke and Schick.

Dadarlat proved that one can calculate the index of a K-homology class $\eta\in K_0(M)$ twisted with an $\epsilon$-flat Hilbert
C*-module bundle $E\to M$ in terms of Lafforgue's \cite{lafforgue-k-theorie-bivariante} $\ell^1$-assembly map $\mu_M^{\ell^1}$
and the almost representation associated to $E$ via parallel transport. The point here is that $\epsilon$ is a number which only
depends on $\eta$ and on a concrete representation of $\mu_M^{\ell^1}(\eta)$ as formal difference of equivalence classes of
projections, but which is independent of the bundle $E$. We will generalize Dadarlat's theorem to the case of almost flat
Fredholm bundles in Theorem~\ref{thm:generalization of Dadarlat's index theorem}.

The key step in the proof of Theorem~\ref{thm:application SNC} and Theorem~\ref{thm:generalization of Dadarlat's index theorem}
is that one can associate to an asymptotically flat Fredholm bundle $(E_n,F_n)_{n\in\N}$ an \emph{asymptotic Fredholm
representation} $(W_n,\rho_n,\hat F_n)_{n\in\N}$ of the fundamental group $\pi_1(M;*)$. Such an asymptotic Fredholm
representation, with respect to a finite presentation $\pi_1(M;*)=\langle L\mid R\rangle$, consists for every $n$ of a graded
Hilbert $B_n$-module $W_n$, an odd operator $\hat F_n$, and a group homomorphism $\rho_n\colon\Fr(L)\to\mathcal L_{B_n}(W_n)$
defined on the free group generated by the set of generators $L$, such that $\rho_n$ and $\hat F_n$ satisfy certain
compatibility conditions, and such that $\|\rho_n(r)-\id\|$ tends to zero as $n\to\infty$ for all relations $r\in R$. If
$(E_n,F_n)_{n\in\N}$ is an asymptotically flat Fredholm bundle, then we take $W_n$ to be the fiber of $E_n$. The homomorphism
$\rho_n$ is defined using parallel transport in $E_n$ as in \cite{connes-gromov-moscovici-conjecture-de-novikov}, and $\hat F_n$
is the restriction of $F_n$ to a fiber.

Let us assume now that all C*-algebras $B_n$ are equal to a single C*-algebra $B$---this reduction step turns out to always be
possible. We will then associate to an asymptotic Fredholm representation its \emph{asymptotic index}
\[
  \asind\left((W_n,\rho_n,\hat F_n)_{n\in\N}\right)\in D(\Sigma C^*\pi_1(M;*),B).
\]
Here $D(-,-)$ denotes Thomsen's D-theory group \cite{thomsen-discrete-asymptotic-homomorphisms}, a variant of Connes's and
Higson's E-theory \cite{connes-higson-deformations-morphismes-asymptotiques}. D-theory is defined in terms of so-called
\emph{discrete asymptotic homomorphisms} which will make it a natural object to work with when we consider asymptotic
representations.

In Theorem~\ref{thm:main theorem} we will show how one can calculate the index $\ind F_n$ for an asymptotically flat Fredholm
bundle $(E_n,F_n)_{n\in\N}$ with underlying C*-algebra $B$ in terms of the asymptotic index of the associated asymptotic
Fredholm representation, at least if $n$ is sufficiently large. This calculation will then be used to prove our applications in
Theorem~\ref{thm:application SNC} and Theorem~\ref{thm:generalization of Dadarlat's index theorem}.

This paper represents part of the author's doctoral dissertation \cite{hunger-asymptotically-flat-fredholm-bundles}
at the Universität Augsburg. The dissertation was
supported by a grant of the \emph{Studienstiftung des deutschen Volkes}, and by the \emph{TopMath} program of the
\emph{Elitenetzwerk Bayern}.

I would like to thank my doctoral advisor Bernhard Hanke for his ideas and encouragement, and for introducing me to the topics
of almost flat bundles and the Strong Novikov Conjecture. I would also like to thank Erik Guentner and Rufus Willett for
inspiring discussions about almost flat bundles and Dadarlat's theorem during my stay at the University of Hawai'i at Mānoa.
Furthermore, I would like to thank Thomas Schick for helpful discussions about E-theory, almost flat bundles, and
almost representations during a stay at the Georg-August-Universität Göttingen.

\section{Almost flat Fredholm bundles}

The key concept in this paper is that of an \emph{almost flat Fredholm bundle}. In order to motivate the definition, consider a
closed Riemannian manifold $M$. Let $E\to M$ be a smooth Hermitian vector bundle (i.\,e.\ a complex vector bundle equipped with
a smoothly varying Hermitian inner product on the fibers), and let $\nabla$ be a connection on $E$ which is compatible
with the metric in the sense that parallel transport preserves the inner product. Finally let $\mathcal R^\nabla$ be the
curvature tensor associated to $\nabla$. Thus,
\[
  \mathcal R^\nabla(X,Y)s=\nabla_X\nabla_Ys-\nabla_Y\nabla_Xs-\nabla_{[X,Y]}s
\]
for smooth sections $X,Y\in\mathcal C^\infty(TM)$ and $s\in\mathcal C^\infty(E)$. We define $\|\mathcal R^\nabla\|$ to be the
supremum of $\|\mathcal R^\nabla(X,Y)s\|$ over all points $p\in M$, all pairs of orthonormal tangent vectors $X,Y\in T_pM$, and
all $s\in E_p$ with $\|s\|=1$. The value $\|\mathcal R^\nabla\|$ has the following geometric significance
\cite[cf.][Proposition~2.7]{hunger-almost-flat-bundles-homological-invariance}:

\begin{prop}\label{prop:parallel transport along the boundary of a disc}
  Let $E$ and $\nabla$ be as above. Let $f\colon D^2\to M$ be a piecewise smooth map, and denote by $T\colon E_{f(1)}\to
  E_{f(1)}$ the parallel transport map along the curve $\tau\mapsto f(e^{2\pi i\tau})$. Then
  \[
    \|T-\id\|\leq\vol(f)\cdot\|\mathcal R^\nabla\|,
  \]
  where
  \[
    \vol(f)=\int_{D^2}\|\partial_sf(s,t)\wedge\partial_tf(s,t)\|d(s,t)
  \]
  is a number which depends only on $f$ and not on $E$ or $\nabla$.\qed
\end{prop}

Now suppose that $M$ is equipped with a smooth triangulation. This means that $M$ is homeomorphic to the geometric realization
of a simplicial complex, and that the simplices are smoothly embedded in $M$. Recall that every point $p\in M$ can then be
written uniquely as a convex combination
\[
  p=\sum_{v\in V_M}\lambda_v(p)\cdot v,
\]
where $V_M$ denotes the set of vertices of $M$. Recall further that the \emph{open star} around a vertex $v\in V_M$ is the set
of all points $p\in M$ with $\lambda_v(p)>0$.

We return to the Hermitian bundle $E\to M$ with compatible connection $\nabla$. For every vertex $v\in V_M$ we consider a
smooth triangulation $\Phi_v\colon S_v\times\C^\ell\to E|_{S_v}$ which is constructed in such a way that $\Phi_v(x,\xi)\in E_x$
is obtained by parallel transporting $\Phi_v(v,\xi)\in E_v$ along the linear path joining $v$ and $x$. For $v,v'\in V_M$ we
consider the transition function $\Psi_{v',v}\colon S_v\cap S_{v'}\to U(\ell)$. This function is defined by the equation
\[
  \Phi_v(x,\xi)=\Phi_{v'}(x,\Psi_{v',v}(x)\xi)
\]
for all $x\in S_v\cap S_{v'}$ and $\xi\in W$. If $x'\in S_v\cap S_{v'}$ is another point, then
$\Psi_{v',v}(x)^{-1}\Psi_{v',v}(x')$ is essentially given by parallel transport along the concatenation $\gamma$ of the
linear paths joining $v$ to $x'$, $x'$ to $v'$, $v'$ to $x$, and $x$ to $v$. Obviously, $\gamma$ bounds a disc whose area is
bounded above by $Cd(x,x')$ for some constant $C>0$ depending only on the Riemannian manifold $M$ and the triangulation. Thus,
Proposition \ref{prop:parallel transport along the boundary of a disc} implies that the transition functions $\Psi_{v',v}$ are
$C\|\mathcal R^\nabla\|$-Lipschitz. In particular, the diameter of the images of $\Psi_{v',v}$ is bounded by $C'\|\mathcal
R^\nabla\|$ where $C'$ does not depend on $E$ or $\nabla$. This motivates the following definition
\cite[cf.][Definition~2.3]{hunger-almost-flat-bundles-homological-invariance}:

\begin{dfn}\label{dfn:almost flat bundles}
  An \emph{$\epsilon$-flat Hermitian vector bundle} over a simplicial complex $X$ is a Hermitian vector bundle $E\to X$,
  together with trivializations $\Phi_v\colon S_v\times\C^\ell\to E|_{S_v}$ over the open stars of $X$, such that each of the
  images of the transition functions $\Psi_{v',v}$ has its diameter bounded by $\epsilon$.
\end{dfn}

Note that in contrast to \cite[Definition~2.3]{hunger-almost-flat-bundles-homological-invariance}, we only demand a bound on the
diameter of the image of $\Psi_{v',v}$, rather than a bound the Lipschitz constants of $\Psi_{v',v}$. Furthermore, we consider
trivializations over the open stars $S_v$ here, rather than trivializations over the simplices of $X$. 

\begin{rem}\label{rem:lipschitz-transition-functions-or-not}
  These changes to the definition actually do not make a big difference. Indeed, as we will recall in
  section~\ref{sec:asymptotic Fredholm representations}, an almost flat bundle (regardless of the precise definition) induces an
  almost representation of the fundamental group of the base space. On the other hand, given an almost representation of the
  fundamental group, by \cite[Theorem~8.7]{hunger-almost-flat-bundles-homological-invariance} one can construct an almost flat bundle
  in the sense of \cite[Definition~2.3]{hunger-almost-flat-bundles-homological-invariance} which has an almost representation close to the given almost
  representation. By \cite[Theorem~8.8]{hunger-almost-flat-bundles-homological-invariance} the newly constructed bundle is isomorphic
  to the bundle that we started with.
\end{rem}

Definition~\ref{dfn:almost flat bundles} can also be extended to the case of Hilbert C*-module bundles
\cite{hunger-almost-flat-bundles-homological-invariance}. Recall from \cite{schick-l2-index-kk-theory-connections} that for a
C*-algebra $B$, a \emph{Hilbert $B$-module bundle} over a space $X$ is a fiber bundle $E\to X$ with typical fiber a Hilbert
$B$-module $W$, and with structure group given by the isometric isomorphisms of $W$. Now as before, we define an $\epsilon$-flat
Hilbert $B$-module bundle over a simplicial complex $X$ to be a Hilbert $B$-module bundle, together with trivializations over
the open stars, such that the diameters of the images of the transition functions is uniformly bounded by $\epsilon$. In
\cite{schick-l2-index-kk-theory-connections}, a definition for connections on Hilbert C*-module bundles is given, and the above
example from Riemannian geometry carries over to Hilbert $B$-module bundles.

We will also need to consider the case of \emph{graded} Hilbert C*-module bundles, where the fibers are $\Z_2$-graded
Hilbert C*-modules, and where the local trivializations are assumed to preserve the grading.

If $W,W'$ are Hilbert $B$-modules, we denote by $\mathcal L_B(W,W')$ the set of adjointable operators $W\to W'$, and by
$\mathcal K_B(W,W')\subset\mathcal L_B(W,W')$ the set of \emph{$B$-compact operators}. As usual, we put $\mathcal
L_B(W)=\mathcal L_B(W,W)$ and $\mathcal K_B(W)=\mathcal K_B(W,W)$. A \emph{generalized Fredholm operator} is an
operator $F\in\mathcal L_B(W,W')$ such that $F^*F-\id$ and $FF^*-\id$ are compact.

If $F\in\mathcal L_B(W,W')$ is a generalized Fredholm operator, then we may consider the \emph{$\Z_2$-graded} Hilbert $B$-module
$\hat W=W\oplus W'$. Then the odd self-adjoint operator
\[
  \hat F=\begin{pmatrix}0&F^*\\F&0\end{pmatrix}\in\mathcal L_B(\hat W)
\]
satisfies $\hat F^2-\id\in\mathcal K_B(\hat W)$. Conversely, if $\hat F\in\mathcal L_B(\hat W)$ is an odd self-adjoint operator
such that $\hat F^2-\id$ is compact, then $\hat F$ is of the form described above, for some generalized Fredholm operator
$F\in\mathcal L_B(\hat W)$. This shows why it can be useful to consider graded Hilbert modules. Motivated by this observation,
we can now formulate the definition of an $\epsilon$-flat Fredholm bundle.

\begin{dfn}
  An \emph{$\epsilon$-flat Fredholm bundle} over a simplicial complex $X$ consists of the following data:
  \begin{itemize}
    \item An $\epsilon$-flat graded Hilbert $B$-module bundle $E\to X$, where $B$ is a unital C*-algebra and where the typical
      fiber of $E$ is a countably generated Hilbert $B$-module $W$,
    \item A map $F_E\colon E\to E$ with the following property: For each vertex $v\in V_X$ there exists a continuous map
      $F_v\colon S_v\to\mathcal L_B(W)$ such that
      \[
        F_E\left(\Phi_v(x,\xi)\right)=\Phi_v\left(x,F_v(x)\xi\right),
      \]
      and such that $F_v$ takes values in the set of odd self-adjoint operators on $W$. Furthermore, we assume that
      $F_v(x)^2-\id\in\mathcal K_B(W)$ for all $x\in S_v$, and that $F_v(x)-F_{v'}(x')\in\mathcal K_B(W)$ for all $v,v'\in V_X$,
      $x\in S_v$, and $x'\in S_{v'}$.
  \end{itemize}
\end{dfn}

The motivating example for almost flat Hilbert module bundles can be extended to give important examples of almost flat
Fredholm bundles. Indeed, let $E\to M$ be a smooth Hilbert $B$-module bundle over a smoothly triangulated manifold, equipped
with a compatible connection $\nabla$, and let $F\colon E\to E$ be a smooth bundle map such that $F^2-\id$ is fiberwise compact,
and such that $F$ commutes with parallel transport on $E$ up to compact operators. We choose the trivializations
$\Phi_v\colon S_v\times W\to E|_{S_v}$ in such a way that for any two vertices $v,v'\in V_M$, the map $\Phi_v(v,-)$ is
obtained from $\Phi_{v'}(v',-)$ by parallel transport along a fixed smooth curve in $M$. With these data, $(E,F)$ becomes a
$C\epsilon$-flat Fredholm bundle over $M$.

\section{The generalized Fredholm index}

In this section, we will review the generalized Fredholm index. The key idea underlying the index is the following: If
$F\in\mathcal L_\C(\mathcal H,\mathcal H')$ is a Fredholm operator between two Hilbert spaces (i.\,e.\ Hilbert $\C$-modules)
$\mathcal H$ and $\mathcal H'$, then the kernel and the cokernel of $F$ are finite-dimensional and the index of $F$ is defined
to be the number $\ind F=\dim\ker F-\dim\coker F\in\Z$. If the underlying C*-algebra is arbitrary, then various complications
arise. Firstly, a generalized Fredholm operator $F\in\mathcal L_B(W,W')$ need not have finitely generated projective kernel and
cokernel, although there always exists a compact perturbation $F'$ of $F$ such that $\ker F'$ and $\coker F'$ are finitely
generated projective Hilbert $B$-modules. Secondly, while finitely generated projective Hilbert $\C$-modules are classified by
their dimension, the same is not true if $\C$ is replaced by an arbitrary unital C*-algebra $B$. However, the K-theory group
$K_0(B)$---which can be defined to be the Grothendieck group of the semigroup of equivalence classes of finitely generated
Hilbert $B$-modules---is a natural target for the index \[ \ind F=[\ker F']-[\coker F']\in K_0(B) \] where $F'-F\in\mathcal
K_B(W,W')$ is compact and $\ker F'$ and $\coker F'$ are finitely generated projective Hilbert $B$-modules.  It is well-known
that $\ind F\in K_0(B)$ does not depend on the choice of compact perturbation $F'$ \cite[Chapter~17]{wegge-olsen-k-theory}.

The generalized Fredholm index fits very well into the framework of Kasparov's KK-theory, in a sense that we will describe next.
Basic references for KK-theory include \cite{blackadar-ktheory-for-operator-algebras,jensen-thomsen-elements-kk-theory}.
KK-theory is a bivariant functor on the category of C*-algebras, but we will only need to consider the case where the first of the
C*-algebras is equal to $\C$. Thus, we will use the notation $KK(B)$ for the group which is usually denoted by $KK(\C,B)$. The
group $KK(B)$ is constructed from what we will call a Kasparov $B$-module.\footnote{This is what is usually called a
Kasparov-$\C$-$B$-bimodule in the literature.} For the rest of this section, we fix a unital C*-algebra $B$.

\begin{dfn}\label{dfn:kasparov modules}
  A \emph{Kasparov $B$-module} is a triple $(W,p,F)$ where
  \begin{itemize}
    \item $W$ is a graded countably generated Hilbert $B$-module,
    \item $p\in\mathcal L_B(W)$ is an even projection, and
    \item $F\in\mathcal L_B(W)$ is an odd operator,
  \end{itemize}
  such that $[p,F],p(F^2-\id),p(F-F^*)\in\mathcal K_B(W)$.

  Denote by $I=[0,1]$ the unit interval.  A \emph{homotopy} of Kasparov $B$-modules is a triple $(W,(p_\tau)_{\tau\in
  I},(F_\tau)_{\tau\in I})$ where $\tau\mapsto p_\tau$ is a continuous path of even projections and $\tau\mapsto F_\tau$ is a
  continuous path of odd operators such that every $(W,p_\tau,F_\tau)$ is a Kasparov $B$-module. The Kasparov $B$-modules
  $(W,p_0,F_0)$ and $(W,p_1,F_1)$ are called \emph{homotopic}.

  A Kasparov $B$-module $(W,p,F)$ is called \emph{degenerate} if $[p,F]=p(F^2-\id)=p(F-F^*)=0$. Two Kasparov $B$-modules
  $(W,p,F)$ and $(W',p',F')$ are called \emph{equivalent} if there exists a unitary equivalence $U\colon W\to W'$ of graded
  Hilbert $B$-modules such that $p'=UpU^*$ and $F'=UFU^*$.
\end{dfn}

By Kasparov's Stabilization Theorem \cite[Theorem~15.4.6]{wegge-olsen-k-theory}, $W\oplus\mathcal H_B$ is unitarily isomorphic
to $\mathcal H_B$, where $\mathcal H_B$ is the standard Hilbert $B$-module.

Thus, every equivalence class of Kasparov $B$-modules has a representative $(W,p,F)$ where $W\subset\mathcal H_B$ is a Hilbert
$B$-submodule. We denote by $\mathcal E(B)$ the set of equivalence classes of Kasparov $B$-modules.
Direct sum gives $\mathcal E(B)$ an abelian monoid structure, with zero element given by $0=[(0,0,0)]$. We consider the
equivalence relation $\sim$ on $\mathcal E(B)$ which is generated by homotopy and the addition of degenerate modules, and put
$KK(B)=\mathcal E(B)/\sim$. It is clear that the monoid structure on $\mathcal E(B)$ induces a monoid structure on $KK(B)$. In
fact, $KK(B)$ is an abelian group \cite[Chapter~17]{blackadar-ktheory-for-operator-algebras}.

There are various different ways in which the definition of the groups $KK(B)$ can be simplified, and the following is the
general setup for these simplifications. Let $S\subset\mathcal E(B)$ be a subsemigroup. By abuse of notation, we will write
$(E,p,F)\in S$ whenever the class of $(E,p,F)$ in $\mathcal E(B)$ is
contained in $S$. We define $\sim_S$ to be the equivalence relation on $S$ which is generated by homotopies
$(W,(p_\tau),(F_\tau))$ such that all $(W,p_\tau,F_\tau)\in S$, and by addition of degenerate modules $(W,p,F)\in S$. Now we
call $S$ \emph{ample} if the natural map $S/\sim_S\to\mathcal E(B)/\sim=KK(B)$, which is induced by the inclusion $S\to\mathcal
E(B)$, is bijective.

Consider the following subsemigroups:
\begin{itemize}
  \item $\mathcal C(B)\subset\mathcal E(B)$ is the set of equivalence classes of Kasparov $B$-modules of the form $(W,p,F)$
    where $F=F^*$ and $\|F\|\leq 1$,
  \item $\mathcal H(B)\subset\mathcal E(B)$ is the set of equivalence classes of Kasparov $B$-modules of the form $(\mathcal
    H_B,p,F)$,\footnote{Note that $\mathcal H_B\oplus\mathcal H_B\cong\mathcal H_B$ by Kasparov's Stabilization Theorem, so that
    indeed $\mathcal H(B)$ is a subsemigroup of $\mathcal E(B)$.}
  \item $\mathcal U(B)\subset\mathcal E(B)$ is the set of equivalence classes of Kasparov $B$-modules of the form $(W,\id,F)$,
  \item $\mathcal Q(B)\subset\mathcal E(B)$ is the set of equivalence classes of Kasparov $B$-modules of the form $(W,p,F)$ with
    $F=F^*=F^{-1}$.
\end{itemize}
The following statement is well-known, and possesses generalizations to the case of $KK(A,B)$ for arbitrary C*-algebras $A$
\cite[cf.][Chapter~17]{blackadar-ktheory-for-operator-algebras}.

\begin{prop}
  The subsemigroups $\mathcal C(B)$, $\mathcal H(B)$, $\mathcal U(B)$, and all of their intersections, are ample. Similarly,
  $\mathcal Q(B)$ and $\mathcal Q(B)\cap\mathcal H(B)$ are ample.\qed
\end{prop}

In order to describe the relation to the generalized Fredholm index, we consider the description of $KK(B)$ corresponding to the
intersection $\mathcal U(B)\cap\mathcal C(B)\cap\mathcal H(B)$. Thus, an element of $KK(B)$ is represented by a triple
$(\mathcal H_B,\id,F)$ for some odd operator $F\in\mathcal L_B(\mathcal H_B)$ such that $F=F^*$, $\|F\|\leq 1$, and such that
$F^2-\id$ is compact. Thus,
\[
  F=\begin{pmatrix}0&F_0^*\\F_0&0\end{pmatrix}
\]
for a generalized Fredholm operator $F_0\colon H_B\to H_B$, where $H_B$ is the ungraded standard Hilbert $B$-module. The
operator $F$ is degenerate if and only if $F_0$ is a unitary isomorphism. The following theorem is well-known.

\begin{thm}[{\cite[Theorem~17.3.11]{wegge-olsen-k-theory}}]
  The map $\ind\colon KK(B)\to K_0(B)$ which associates to $[\mathcal H_B,\id,F]\in KK(B)$ the index $\ind F_0\in K_0(B)$ is a
  group isomorphism.\qed
\end{thm}

Let $p\in M_n(B)$ be a projection. Then one easily calculates that $\ind[pB^n\oplus 0,\id,0]=[p]\in K_0(B)$. This description of
pre-images implies the following recognition principle for the generalized Fredholm index.

\begin{corl}\label{corl:recognition principle for the generalized Fredholm index}
  If $\ind'\colon KK(B)\to K_0(B)$ is a group homomorphism which satisfies $\ind'[pB^n\oplus 0,\id,0]=[p]$ for all projections
  $p\in M_n(B)$ and $n\in\N$, then $\ind=\ind'$.\qed
\end{corl}

We will next use Corollary \ref{corl:recognition principle for the generalized Fredholm index} to give a new
description of the generalized Fredholm index. In order to do this, we use the description of $KK(B)$ via the ample submonoid
$\mathcal Q(B)\cap\mathcal H(B)$. Thus, we represent an element of $KK(B)$ by a triple $(\mathcal H_B,p,F)$ where
$F=F^*=F^{-1}$. Consequently, $F\in\mathcal L_B(\mathcal H_B)$ is an odd self-adjoint unitary, $p\in\mathcal L_B(\mathcal H_B)$
is an even projection, and $[F,p]\in\mathcal K_B(\mathcal H_B)$. Such a triple is degenerate if and only if $[F,p]=0$.

To any odd self-adjoint unitary $F\in\mathcal L_B(\mathcal H_B)$ we associate the C*-algebra
\[
  Q_F=\left\{x\in\mathcal L_B^{\mathrm{ev}}(\mathcal H_B):[F,x]\in\mathcal K_B(\mathcal H_B)\right\}.
\]
Of course, $(\mathcal H_B,p,F)$ is a Kasparov $B$-module if and only if $p$ is a projection in $Q_F$. We can write
\[
  F=\begin{pmatrix}0&F_0^*\\F_0&0\end{pmatrix}\in\mathcal L_B(\mathcal H_B)=\mathcal L_B(H_B\oplus H_B)
\]
for a unitary operator $F_0\in\mathcal L_B(H_B)$. An element in $\mathcal L_B^{\mathrm{ev}}(\mathcal H_B)$ is of the form
$x=x_0\oplus x_1$ for operators $x_0,x_1\in\mathcal L_B(H_B)$, and $x_0\oplus x_1\in Q_F$ if and only if
$F_0x_0F_0^*-x_1\in\mathcal K_B(H_B)$. This shows that there is a split short exact sequence
\[
  \begin{tikzpicture}
    \matrix (m) [matrix of math nodes, row sep=3em, column sep=3em, commutative diagrams/every cell]
    {
      0 & \mathcal K_B(H_B) & Q_F & \mathcal L_B(H_B) & 0 \\
    };
    \path[-stealth, commutative diagrams/.cd, every label] (m-1-1) edge (m-1-2) (m-1-2) edge node [above] {$i_F$} (m-1-3)
      (m-1-3) edge node [above] {$\pi_F$} (m-1-4) (m-1-4) edge [bend left=30] node [below] {$s_F$} (m-1-3) edge (m-1-5);
  \end{tikzpicture}
\]
of C*-algebras, where the maps are given by $i_F(x)=x\oplus 0$, $\pi_F(x\oplus y)=y$, and $s_F(y)=F_0^*yF_0\oplus y$. Since the
functor $K_0$ is split exact, there is an associated split exact sequence
\[
  \begin{tikzpicture}
    \matrix (m) [matrix of math nodes, row sep=3em, column sep=3em, commutative diagrams/every cell]
    {
      0 & K_0(\mathcal K_B(H_B)) & K_0(Q_F) & K_0(\mathcal L_B(H_B)) & 0 \\
    };
    \path[-stealth, commutative diagrams/.cd, every label] (m-1-1) edge (m-1-2) (m-1-2) edge node [above] {$(i_F)_*$} (m-1-3)
      (m-1-3) edge node [above] {$(\pi_F)_*$} (m-1-4) (m-1-4) edge [bend left=30] node [below] {$(s_F)_*$} (m-1-3) edge (m-1-5);
  \end{tikzpicture}
\]
of K-theory groups. Then $(\pi_F)_*(\id-(s_F)_*(\pi_F)_*)=0$, so there exists a unique group homomorphism $\rho_F\colon
K_0(Q_F)\to K_0(\mathcal K_B(H_B))$ with $(i_F)_*\rho_F=\id-(s_F)_*(\pi_F)_*$. Now for each projection $p\in Q_F$ we put
\[
  \ind_p(F)=\rho_F[p]\in K_0(\mathcal K_B(H_B))\cong K_0(B).
\]
We define $\ind'[\mathcal H_B,p,F]=\ind_p(F)$ for $(\mathcal H_B,p,F)\in\mathcal Q(B)\cap\mathcal H(B)$. The main result of this
section will be that $\ind'\colon KK(B)\to K_0(B)$ is a well-defined group homomorphism and in fact that $\ind=\ind'$. We begin
with a very useful easy observation.

\begin{lem}\label{lem:homomorphism associated to a split exact sequence of abelian groups}
  Let
  \[
    \begin{tikzpicture}
      \matrix (m) [matrix of math nodes, row sep=3em, column sep=3em, commutative diagrams/every cell]
      {
        0 & A & B & C & 0 \\
        0 & A' & B' & C' & 0 \\
      };
      \path[-stealth, commutative diagrams/.cd, every label] (m-1-1) edge (m-1-2) (m-1-2) edge node [above] {$i$} (m-1-3) edge
        node [left] {$f|_A$} (m-2-2) (m-1-3) edge node [above] {$\pi$} (m-1-4) edge node [left] {$f$} (m-2-3) (m-1-4) edge
        (m-1-5) edge node [right] {$\bar f$} (m-2-4) edge [bend left=30] node [below] {$s$} (m-1-3) (m-2-1) edge (m-2-2) (m-2-2)
        edge node [above] {$i'$} (m-2-3) (m-2-3) edge node [above] {$\pi'$} (m-2-4) (m-2-4) edge [bend left=30] node [below]
        {$s'$} (m-2-3) edge (m-2-5);
    \end{tikzpicture}
  \]
  be a commutative diagram of split short exact sequences of abelian groups. Let $\rho\colon B\to A$ be the unique homomorphism
  with $i\rho=\id-s\pi$, and let $\rho'\colon B'\to A'$ be the unique homomorphism with $i'\rho'=\id-s'\pi'$. Then
  $\rho'f=f|_A\rho$.
\end{lem}
\begin{proof}
  We have $i'\rho'f=f-s'\pi'f=f-s'\bar f\pi=f-fs\pi=fi\rho=i'f|_A\rho$. The claim follows because $i'$ is injective.
\end{proof}

By applying Lemma~\ref{lem:homomorphism associated to a split exact sequence of abelian groups} several times, one can prove
straightforwardly that the map $\ind'$ is well-defined. For example, let us show that $\ind_{U^*pU}(U^*FU)=\ind_p(F)$ whenever
$U\in\mathcal L_B(\mathcal H_B)$ is an even unitary.

Since $U$ is even, we can write $U=U_0\oplus U_1$ for unitaries $U_0,U_1\in\mathcal L_B(H_B)$. Let $\Ad_{U_k}\colon\mathcal
L_B(H_B)\to\mathcal L_B(H_B)$ be the maps given by $\Ad_{U_k}(x)=U_k^*xU_k$. Then we have a commutative diagram
\[
  \begin{tikzpicture}
    \matrix (m) [matrix of math nodes, row sep=4em, column sep=4em, commutative diagrams/every cell] {
      0 & \mathcal K_B(H_B) & Q_F & \mathcal L_B(H_B) & 0 \\
      0 & \mathcal K_B(H_B) & Q_{U^*FU} & \mathcal L_B(H_B) & 0 \\
    };
    \path[-stealth, commutative diagrams/.cd, every label]
      (m-1-1) edge (m-1-2) (m-1-2) edge node [above] {$i_F$} (m-1-3)
      (m-1-3) edge node [above] {$\pi_F$} (m-1-4) (m-1-4) edge (m-1-5)
      (m-1-4) edge [bend left=30] node [below] {$s_F$} (m-1-3)
      (m-2-1) edge (m-2-2) (m-2-2) edge node [above] {$i_{U^*FU}$} (m-2-3)
      (m-2-3) edge node [above] {$\pi_{U^*FU}$} (m-2-4) (m-2-4) edge (m-2-5)
      (m-2-4) edge [bend left=30] node [below] {$s_{U^*FU}$} (m-2-3)
      (m-1-2) edge node [left] {$\Ad_{U_0}$} (m-2-2) (m-1-3) edge node [left] {$\Ad_{U_0}\oplus\Ad_{U_1}$} (m-2-3)
      (m-1-4) edge node [right] {$\Ad_{U_1}$} (m-2-4);
  \end{tikzpicture}
\]
of split short exact sequences of C*-algebras. Let $\rho\colon K_0(Q_F)\to K_0(\mathcal K_B(H_B))$ and
$\rho'\colon K_0(Q_{U^*FU})\to K_0(\mathcal K_B(H_B))$ be the respective splitting homomorphisms. Then Lemma~\ref{lem:homomorphism associated to a split exact sequence of abelian groups} shows that
\begin{align*}
  (\Ad_{U_0})_*\ind_p(F)&=(\Ad_{U_0})_*\rho[p]=\rho'(\Ad_{U_0}\oplus\Ad_{U_1})_*[p]\\
                        &=\rho'[U^*pU]=\ind_{U^*pU}(U^*FU).
\end{align*}
The claim follows from the fact that $(\Ad_{U_0})_*$ equals the identity on $K_0(\mathcal K_B(H_B))$.

\begin{thm}\label{thm:ind=ind'}
  The maps $\ind\colon KK(B)\to K_0(B)$ and $\ind'\colon KK(B)\to K_0(B)$ coincide.
\end{thm}
\begin{proof}
  By Corollary~\ref{corl:recognition principle for the generalized Fredholm index} it suffices to show that
  $\ind'[pB^n\oplus 0,\id,0]=[p]$ if $p\in M_n(B)$ is any projection. Therefore, we have to represent $[pB^n\oplus 0,\id,0]$ by a
  Kasparov $B$-module in $\mathcal Q(B)\cap\mathcal H(B)$. The construction goes as follows: Firstly, we add on the degenerate
  module $((1-p)B^n\oplus 0,0,0)$ to obtain $(B^n\oplus 0,p,0)$. Next, add on $(0\oplus B^n,0\oplus 0,0)$ and perturb
  compactly to obtain
  \[
    \left(B^n\oplus B^n,p\oplus 0,\begin{pmatrix}0&1\\1&0\end{pmatrix}\right).
  \]
  Finally, stabilize by adding $(\mathcal H_B,0,\begin{psmallmatrix}0&1\\1&0\end{psmallmatrix})$, and obtain
  \[
    \left((B^n\oplus H_B)\oplus(B^n\oplus H_B),(p\oplus 0)\oplus 0,\begin{pmatrix}0&1\\1&0\end{pmatrix}\right).
  \]
  This is equivalent to $(\mathcal H_B,p'\oplus 0,F)$ for some $p'\in\mathcal K_B(H_B)$ such that $[p']=[p]\in K_0(B)$. Now in
  the short exact sequence
  \[
    \begin{tikzpicture}
      \matrix (m) [matrix of math nodes, row sep=3em, column sep=3em, commutative diagrams/every cell]
      {
        0 & K_0(\mathcal K_B(H_B)) & K_0(Q_F) & K_0(\mathcal L_B(H_B)) & 0 \\
      };
      \path[-stealth, commutative diagrams/.cd, every label]
        (m-1-1) edge (m-1-2) (m-1-2) edge node [above] {$(i_F)_*$} (m-1-3) (m-1-3) edge node [above] {$(\pi_F)_*$} (m-1-4)
        (m-1-4) edge (m-1-5) edge [bend left=30] node [below] {$(s_F)_*$} (m-1-3);
    \end{tikzpicture}
  \]
  we have $(i_F)_*[p']=[p'\oplus 0]=(\id-s_F\pi_F)_*[p'\oplus 0]=(i_F)_*\ind_{p'\oplus 0}(F)$, and therefore
  $\ind'[pB^n\oplus 0,\id,0]=\ind_{p'\oplus 0}(F)=[p']=[p]$ as claimed.
\end{proof}

\section{The index bundle}

In this section, we will review the definition of the index of a bundle of Fredholm operators. We will then use
Theorem~\ref{thm:ind=ind'} in order to give a different description of this index.

The general setup is as follows: Let $X$ be a compact Hausdorff space, and let $p\colon E\to X$, $p'\colon E'\to X$ be Hilbert
$B$-module bundles. A \emph{Hilbert $B$-module bundle morphism} from $E$ to $E'$ is a map $F\colon E\to E'$ between the total
spaces such that $p'\circ F=p$, and such that for all local trivializations $\Phi\colon U\times W\to E|_U$ and $\Phi'\colon
U'\times W'\to E'|_{U'}$ there exists a continuous map $f\colon U\cap U'\to\mathcal L_B(W,W')$ with
\[
  f\Phi(x,\xi)=\Phi'(x,f(x)\xi)
\]
for all $x\in U\cap U'$ and $\xi\in W$. We denote by $\mathcal L_B(E,E')$ the set of Hilbert $B$-module bundle morphisms $E\to
E'$, and abbreviate $\mathcal L_B(E)=\mathcal L_B(E,E)$. By definition, $F\in\mathcal L_B(E,E')$ restricts to adjointable maps
$F_x\in\mathcal L_B(E_x,E_x')$ on the fibers, and taking fiberwise adjoints defines a morphism $F^*\in\mathcal L_B(E',E)$.
Obviously, $\id^*=\id$ and $(GF)^*=F^*\circ G^*$. Of course, $F\colon E\to E'$ is an isometric isomorphism if and only if
$F\in\mathcal L_B(E,E')$, $FF^*=\id$, and $F^*F=\id$.

We will also need to consider the Hilbert $C(X;B)$-module $\Gamma(E)$ of continuous sections of $E\to X$. Postcomposition with
$F\in\mathcal L_B(E,E')$ defines an adjointable operator $F_*\in\mathcal
L_{C(X;B)}(\Gamma(E),\Gamma(E'))$, and its adjoint is given by postcomposition with $F^*$. Furthermore, $G_*F_*=(GF)_*$ and
$\id_*=\id$. We will need the following observation, which is a simple consequence of compactness of $X$:

\begin{lem}\label{lem:fiberwise compact operators}
  Let $p\colon E\to X$ and $p'\colon E'\to X$ be Hilbert $B$-module bundles over a compact Hausdorff space $X$, and let $G\colon
  E\to E'$ be a map such that every $x\in X$ possesses a neighborhood $U_x\subset X$, local
  trivializations $\Phi\colon U_x\times W_x\to E|_{U_x}$,
  $\Phi'\colon U_x'\times W_x'\to E'|_{U_x'}$, and a continuous map $G_x\colon U_x\to\mathcal K_B(W_x,W_x')$ with
  \[
    G(\Phi_x(y,\xi))=\Phi_x'(y,G_x(y)\xi)
  \]
  for all $y\in U_x$ and all $\xi\in W_x$. Then $G_*\in\mathcal K_{C(X;B)}(\Gamma(E),\Gamma(E'))$.\qed
\end{lem}

\begin{dfn}
  A \emph{bundle of (generalized) Fredholm operators} over $X$ consists of a graded Hilbert $B$-module bundle $E\to X$, where
  $B$ is a unital C*-algebra and the fibers of $E$ are countably generated, and of an odd self-adjoint bundle morphism $F\in\mathcal
  L_B(E)$ with the property that $F^2-\id$ restricts to compact operators on all fibers.
\end{dfn}

It is clear that an $\epsilon$-flat Fredholm bundle, as defined before, is a bundle of Fredholm operators. Now if $(E,F)$ is a
bundle of Fredholm operators, then Lemma~\ref{lem:fiberwise compact operators} implies that $F_*^2-\id\in\mathcal
K_{C(X;B)}(\Gamma(E))$.

\begin{prop}
  Let $(E,F)$ be a bundle of Fredholm operators over a compact Hausdorff base space $X$. Then the triple $\hat
  E=(\Gamma(E),\id,F_*)$ is a Kasparov $C(X;B)$-module in the sense of definition \ref{dfn:kasparov modules}, and therefore
  defines a class $[\hat E]\in KK(C(X;B))$. We define
  \[
    \ind F=\ind[\hat E]\in K_0(C(X;B)).
  \]
\end{prop}
\begin{proof}
  Since $F_*$ is odd and self-adjoint, and since $F_*^2-\id$ is compact, it only remains to prove that $\Gamma(E)$ is countably
  generated, which again is a straightforward consequence of the compactness of $X$.
\end{proof}

The index $\ind F$ can be described as follows: Since $F\in\mathcal L_B(E)$ is odd and self-adjoint, we can write
\[
  F=\begin{pmatrix}0&F_0^*\\F_0&0\end{pmatrix}\in\mathcal L_B\left(E^{(0)}\oplus E^{(1)}\right)
\]
with respect to the grading $E=E^{(0)}\oplus E^{(1)}$. Thus $F_0\in\mathcal L_B(E^{(0)},E^{(1)})$. One can now perturb $F_0$
compactly to a partial isometry $F_1\in\mathcal L_B(E^{(0)},E^{(1)})$, so that $p=\id-F_1^*F_1$ and $q=\id-F_1F_1^*$ are
fiberwise projections, whose images then form bundles $E_p=pE^{(0)}$ and $E_q=qE^{(1)}$ of finitely generated projective Hilbert
$B$-modules over $X$. Now the definition of the generalized Fredholm index implies that
\[
  \ind F=[E_p]-[E_q],
\]
thus recovering the definition of the index bundle introduced by Jänich
\cite{jaenich-vektorraumbuendel-raum-der-fredholm-operatoren}.

By Theorem~\ref{thm:ind=ind'}, we also have $\ind F=\ind'[\hat E]\in K_0(C(X;B))$. In the remainder of this section, we will use
this observation in order to give yet another description of the index of an $\epsilon$-flat Fredholm bundle over a simplicial
complex. In order to do this, it is useful to modify the Kasparov $C(X;B)$-module $\hat E$ in such a way that the relevant
index-theoretic information is completely contained in the projection entry. This is done by a well-known procedure
\cite[cf.][Section~17]{blackadar-ktheory-for-operator-algebras}, which we will describe next.

We fix a unital C*-algebra $B$, a countably generated graded Hilbert $B$-module $W$, and an
odd operator $F\in\mathcal L_B(W)$. Let $\phi\colon\R\to\R$ be the function given by
\[
  \phi(t)=\begin{cases}-1,&t\leq -1,\\t,&-1\leq t\leq 1,\\1,&t\geq 1.\end{cases}
\]
Since $\frac 12(F+F^*)$ is self-adjoint, we may define another operator $C(F)=\phi(\frac 12(F+F^*))$ by continuous functional
calculus in the C*-algebra $\mathcal L_B(W)$. It is easy to see that $C(F)$ is an odd self-adjoint operator with $\|C(F)\|\leq
1$. We put
\[
  Q(F)=\begin{pmatrix}C(F)&\sqrt{1-C(F)^2}\\\sqrt{1-C(F)^2}&-C(F)\end{pmatrix}\in\mathcal L_B(W\oplus W^{\mathrm{op}}),
\]
which is an odd self-adjoint unitary,\footnote{The graded Hilbert $B$-module $W^{\mathrm{op}}$ is given by the Hilbert $B$-module
  $W$, but with opposite grading.} and
\[
  U(F)=\frac 1{\sqrt 2}\begin{pmatrix}1&Q(F)\\Q(F)&-1\end{pmatrix}\in\mathcal L_B\left((W\oplus W^{\mathrm{op}})\oplus (W\oplus
  W^{\mathrm{op}})^{\mathrm{op}}\right),
\]
which is an even self-adjoint unitary. The crucial properties of this construction are summarized in the following proposition
whose proof is straightforward.

\begin{prop}\label{prop:properties of U(F)}
  Let $W$ and $F$ be as above, and put
  \[
    T=\begin{pmatrix}0&1\\1&0\end{pmatrix}\in\mathcal L_B\left((W\oplus W^{\mathrm{op}})\oplus(W\oplus
    W^{\mathrm{op}})^{\mathrm{op}}\right).
  \]
  Assume furthermore that $H\in\mathcal L_B(W)$ is such that $[H,F]$, $[H,F^*]$, $H(F^2-\id)$, and $H(F^*-F)$ are all contained
  in $\mathcal K_B(W)$. Then also the operators
  \[
    \left((H\oplus 0)\oplus 0\right)\cdot\left(U(F)^*TU(F)-(F\oplus(-F))\oplus((-F)\oplus F)\right)
  \]
  and
  \[
    \left(U(F)^*TU(F)-(F\oplus(-F))\oplus((-F)\oplus F)\right)\cdot\left((H\oplus 0)\oplus 0\right)
  \]
  are compact. In particular, the commutator $[(H\oplus 0)\oplus 0,U(F)^*TU(F)]$ is compact.\qed
\end{prop}

Now let $B$ be a unital C*-algebra, and let $(E,F_E)$ be an $\epsilon$-flat Fredholm bundle over a finite simplicial complex $X$,
where the fiber $W$ of $E$ is a countably generated Hilbert $B$-module. Let $V_X=\{v_1,\ldots,v_n\}\subset X$ be the (finite)
set of vertices of $X$. We use the abbreviation $S_k=S_{v_k}$ for the open star around $v_k$. Consider the local trivializations
$\Phi_k=\Phi_{v_k}\colon S_k\times W\to E|_{S_k}$ and the transition functions $\Psi_{jk}=\Psi_{v_j,v_k}\colon S_j\cap
S_k\to\mathcal L_B(W)$. Thus,
\[
  \Phi_k(x,\xi)=\Phi_j(x,\Psi_{jk}(x)\xi)
\]
for all $x\in S_j\cap S_k$ and $\xi\in W$. Similarly, we put $F_k=F_{v_k}\colon S_k\to\mathcal L_B(W)$, so that
\[
  F_E\left(\Phi_k(x,\xi)\right)=\Phi_k\left(x,F_k(x)\xi\right)
\]
for all $x\in S_k$ and $\xi\in W$. Further, we consider the operator $F=F_1(v_1)\in\mathcal L_B(W)$.

We fix an even isometric isomorphism $U\colon W\oplus\mathcal H_B\to\mathcal H_B$ (which exists by Kasparov's Stabilization
Theorem). Now we define
\[
  \Psi_{jk}'(x)=U\left(\Psi_{jk}(x)\oplus 0\right)U^*\in\mathcal L_B(\mathcal H_B)
\]
for $x\in S_j\cap S_k$, and
\[
  F'=U\left(F\oplus\id\right)U^*\in\mathcal L_B(\mathcal H_B).
\]
Using the assumption that $F-F_k(x)$ is compact for all $k$ and all $x\in S_k$, one obtains

\begin{lem}\label{lem:[Psi',F'] is compact}
  For all $j$, all $k$, and all $x\in S_j\cap S_k$, the operator $[\Psi_{jk}'(x),F']$ is compact.\qed
\end{lem}

We write $\mathcal H_B'=(\mathcal H_B\oplus\mathcal H_B^{\mathrm{op}})\oplus(\mathcal H_B\oplus\mathcal
H_B^{\mathrm{op}})^{\mathrm{op}}$. Then the construction described above yields an even self-adjoint unitary $U(F')\in\mathcal
L_B(\mathcal H_B')$, and we define
\[
  \Psi_{jk}''(x)=U(F')\left((\Psi_{jk}'(x)\oplus 0)\oplus 0\right)U(F')^*\in\mathcal L_B(\mathcal H_B')
\]
for $x\in S_j\cap S_k$. Let $T=\begin{psmallmatrix}0&1\\1&0\end{psmallmatrix}\in\mathcal L_B(\mathcal H_B')$ be as in
Proposition~\ref{prop:properties of U(F)}. Since $[\Psi_{jk}'(x),F']$ is compact by Lemma~\ref{lem:[Psi',F'] is compact}, also
$[U(F')^*\Psi_{jk}''(x)U(F'),(F'\oplus(-F'))\oplus((-F')\oplus F')]$ is compact, and
Proposition~\ref{prop:properties of U(F)} implies that
\[
  \left[U(F')^*\Psi_{jk}''(x)U(F'),U(F')^*TU(F')\right]
\]
is compact as well. Thus, $[\Psi_{jk}''(x),T]\in\mathcal K_B(\mathcal H_B')$ for all $x\in S_j\cap S_k$.

We fix an even isometric isomorphism $V\colon\mathcal H_B'\to\mathcal H_B$ (which again exists since $\mathcal H_B'$ is countably
generated), and consider $T'=VTV^*\in\mathcal L_B(\mathcal H_B)$. Then $T'$ is an odd self-adjoint unitary, and we may consider
the C*-algebra
\[
  Q=Q_{T'}=\left\{x\in\mathcal L_B^{\mathrm{ev}}(\mathcal H_B):[x,T']\in\mathcal K_B(\mathcal H_B)\right\}
\]
as in the definition of $\ind'$. Since we have seen that $[\Psi_{jk}''(x),T]$ is compact, it follows that
\[
  \tilde P^E(x)=\left(\sqrt{\lambda_j(x)\lambda_k(x)}(V\Psi_{jk}''(x)V^*)\right)_{j,k}
\]
is contained in $M_n(Q)$ for all $x\in X$, where the $\lambda_j$ are the barycentric coordinates on $X$. In addition, $\tilde
P^E(x)$ is self-adjoint because $\Psi_{jk}''(x)^*=\Psi_{kj}''(x)$ for all $j,k$. Finally, a calculation shows that $\tilde
P^E(x)^2=\tilde P^E(x)$, so that $\tilde P^E(x)$ is a projection. Thus, $\tilde P^E$ defines a projection in $C(X;M_n(Q))\cong
M_n(C(X;Q))$, and hence a class $[\tilde P^E]\in K_0(C(X;Q))\cong K_0(C(X)\otimes Q)$. We have already seen that there is a
split short exact sequence
\[
  \begin{tikzpicture}
    \matrix (m) [matrix of math nodes, row sep=3em, column sep=3em, commutative diagrams/every cell]
    {
      0 & \mathcal K_B(H_B) & Q & \mathcal L_B(H_B) & 0, \\
    };
    \path[-stealth, commutative diagrams/.cd, every label] (m-1-1) edge (m-1-2) (m-1-2) edge node [above] {$i_{T'}$} (m-1-3)
      (m-1-3) edge node [above] {$\pi_{T'}$} (m-1-4) (m-1-4) edge [bend left=30] node [below] {$s_{T'}$} (m-1-3) edge (m-1-5);
  \end{tikzpicture}
\]
so the sequence
\begingroup\small
\[
  \begin{tikzpicture}
    \matrix (m) [matrix of math nodes, row sep=3em, column sep=3em, commutative diagrams/every cell]
    {
      0 & C(X)\otimes\mathcal K_B(H_B) & C(X)\otimes Q & C(X)\otimes\mathcal L_B(H_B) & 0, \\
    };
    \path[-stealth, commutative diagrams/.cd, every label] (m-1-1) edge (m-1-2) (m-1-2) edge node [above] {$\id\otimes i_{T'}$}
      (m-1-3) (m-1-3) edge node [above] {$\id\otimes\pi_{T'}$} (m-1-4) (m-1-4) edge [bend left=30] node [below] {$\id\otimes
      s_{T'}$} (m-1-3) edge (m-1-5);
  \end{tikzpicture}
\]
\endgroup
is split exact as well because $C(X)$ is nuclear. Thus, we get an associated split exact sequence in K-theory, and therefore a
group homomorphism $\rho_X\colon K_0(C(X)\otimes Q)\to K_0(C(X)\otimes\mathcal K_B(H_B))$ such that $(\id\otimes
i_{T'})_*\circ\rho_X=\id-(\id\otimes s_{T'}\pi_{T'})_*$. Of course, $C(X)\otimes\mathcal K_B(H_B)$ is naturally isomorphic to
$C(X;B)\otimes\mathcal K$, so in particular there is a natural isomorphism $K_0(C(X)\otimes\mathcal K_B(H_B))\cong
K_0(C(X;B)\otimes\mathcal K)\cong K_0(C(X;B))$. Now we obtain the following description of the index of the $\epsilon$-flat
Fredholm bundle $(E,F_E)$.

\begin{thm}\label{thm:alternative calculation of the index of an almost flat Fredholm bundle}
  Under the above isomorphisms, $\ind F_E=\rho_X[\tilde P^E]\in K_0(C(X;B))$.
\end{thm}
\begin{proof}
  From the definition and from Theorem~\ref{thm:ind=ind'} we have that $\ind F_E=\ind'[\hat E]$. Thus, we will first
  replace $\hat E$ by a Kasparov $C(X;B)$-module in $\mathcal Q(C(X;B))\cap\mathcal H(C(X;B))$ in order to be able to calculate
  $\ind'[\hat E]$. In fact, consider
  \[
    T_X'=\id_X\times(T'\oplus\cdots\oplus T')\in\mathcal L_B\left(X\times(\mathcal H_B\oplus\cdots\oplus\mathcal H_B)\right)
  \]
  where $T'\in\mathcal L_B(\mathcal H_B)$ is as in the definition of the C*-algebra $Q$.

  One can now show that the triple
  \[
    \hat E'=\left(\Gamma(X\times(\mathcal H_B\oplus\cdots\oplus\mathcal H_B)),(\tilde P^E)_*,(T_X')_*\right)
  \]
  is a Kasparov $C(X;B)$-module in $\mathcal Q(C(X;B))\cap\mathcal H(C(X;B))$, and that $[\hat E]=[\hat E']\in KK(C(X;B))$.
  Thus, $\ind F_E=\ind'[\hat E']\in K_0(C(X;B))$.

  It is then straightforward to relate this result to $\rho_X[\tilde P^E]$. For details we refer to
  \cite[Theorem~4.5.7]{hunger-asymptotically-flat-fredholm-bundles}.
\end{proof}

\section{Asymptotic Fredholm representations}

\label{sec:asymptotic Fredholm representations}

Up to now, we have not used that $\epsilon$ in the definition of an $\epsilon$-flat Fredholm bundle is actually small. In
this section, we will consider $\epsilon$-flat Fredholm bundles where $\epsilon$ tends to zero in the sense of the following
definition.

\begin{dfn}
  An \emph{asymptotically flat Fredholm bundle} is a sequence of $\epsilon_n$-flat Fredholm bundles $(E_n,F_n)$ with
  the same underlying unital C*-algebra $B$, such that $\lim_{n\to\infty}\epsilon_n=0$.
\end{dfn}

If $E\to M$ is a flat smooth Hilbert $B$-module bundle with connection over a smooth manifold $M$, then parallel transport along
a curve $\gamma$ is invariant under homotopy of $\gamma$ relative to its endpoints. This construction yields a representation of
the fundamental group of $M$ on the fibers of $E$. It turns out that a similar construction is feasible for asymptotically flat
Fredholm bundles as well. This yields the concept of an almost representation associated to an almost flat bundle. This
connection between almost flat bundles and almost representations has first been observed by Connes, Gromov, and Moscovici
\cite{connes-gromov-moscovici-conjecture-de-novikov}, and was further analyzed by Manuilov and Mishchenko
\cite{manuilov-mishchenko-almost-asymptotic-fredholm-representations}, Mishchenko and Teleman
\cite{mishchenko-teleman-almost-flat-bundles}, Hanke \cite{hanke-positive-scalar-curvature}, Carrión and Dadarlat
\cite{carrion-dadarlat-almost-flat-k-theory}, and the author of this paper
\cite{hunger-almost-flat-bundles-homological-invariance}.

In order to describe the construction of the almost representation associated to an almost flat bundle, let $E\to X$ be an
$\epsilon$-flat Hilbert $B$-module bundle over a simplicial complex $X$, with typical fiber $W$. Consider vertices $v_0$ and
$v_1$ which bound an edge in $X$. In particular, $\frac 12(v_0+v_1)\in S_{v_0}\cap S_{v_1}$, and we may define the
\emph{transport operator} along the (oriented) edge $[v_0,v_1]$ to be the even unitary operator
$T_{(v_0,v_1)}=\Psi_{v_1,v_0}(\frac 12(v_0+v_1))\in\mathcal L_B(W)$.

We can also define transport along \emph{simplicial paths} in $X$. A simplicial path in $X$ is a finite sequence
$\Gamma=(v_0,v_1,\ldots,v_k)$ of vertices in $X$, such that $[v_i,v_{i+1}]$ is an edge for all $i$. One should imagine $\Gamma$
as the continuous path which is the concatenation of the linear paths joining $v_i$ and $v_{i+1}$. We define the transport
operator along such a simplicial path $\Gamma$ to be
\[
  T_\Gamma=T_{(v_{n-1},v_n)}\circ\cdots\circ T_{(v_0,v_1)}\in\mathcal L_B(W).
\]
Of course, each $T_{(v_0,v_1)}$ and hence also each $T_\Gamma$ is an isometric automorphism of $W$. Furthermore, $T_\Gamma$ is
graded if $E$ is a graded almost flat Hilbert $B$-module bundle.

If $\Gamma=(v_0,\ldots,v_k)$ and $\Gamma'=(v_k,\ldots,v_{k+\ell})$ are two simplicial paths such that the endpoint of $\Gamma$
equals the starting point of $\Gamma'$, then the concatenation of $\Gamma$ and $\Gamma'$ is the simplicial path
$\Gamma'*\Gamma=(v_0,\ldots,v_{k+\ell})$. It is clear from the definition that $T_{\Gamma'*\Gamma}=T_{\Gamma'}T_\Gamma$. The
\emph{opposite} of a simplicial path $\Gamma=(v_0,\ldots,v_n)$ is the simplicial path $\bar\Gamma=(v_n,\ldots,v_0)$, and clearly
$T_{\bar\Gamma}=T_\Gamma^*$.

The most important feature of the transport operators is that they satisfy an analogue of Proposition \ref{prop:parallel
transport along the boundary of a disc}.

\begin{thm}[{\cite[Theorem 3.4]{hunger-almost-flat-bundles-homological-invariance}}]
  \label{thm:transport along contractible loops}
  Let $\Gamma=(v_0,\ldots,v_n)$ be a contractible simplicial loop in $X$, in the sense that the concatenation of linear
  paths joining $v_i$ and $v_{i+1}$ form a loop in $X$ which is contractible. Then there are constants $C(\Gamma)>0$ and
  $\delta(\Gamma)>0$, depending only on $X$ and $\Gamma$, such that the following holds:

  Let $E\to X$ be an $\epsilon$-flat Hilbert $B$-module bundle where $\epsilon\leq\delta$. Then the transport operator
  $T_\Gamma\in\mathcal L_B(W)$ satisfies $\|T_\Gamma-\id\|\leq C\cdot\epsilon$.\qed
\end{thm}

Now suppose that $X$ is connected and that $\pi_1(X;v_0)=\langle L\mid R\rangle$ is finitely presented, where $v_0\in X$ is a
base vertex. This means that $L\subset\pi_1(X;v_0)$ is a finite subset which generates $\pi_1(X;v_0)$, and that $R\subset\Fr(L)$
is a finite subset of the free group generated by $L$ such that the normal subgroup $\langle R\rangle\subset\Fr(L)$ generated by
$R$ equals the kernel of the projection $\Fr(L)\to\pi_1(X;v_0)$.

It follows from the simplicial approximation theorem that every $g\in L$ can be represented by a simplicial loop
$\Gamma_g$ in $X$, based at $v_0$. We fix a choice of these simplicial loops. Now suppose that $E\to X$ is an $\epsilon$-flat
Hilbert $B$-module bundle over $X$, with typical fiber $W$. We consider the group homomorphism $\rho_E\colon\Fr(L)\to U(W)$
which is uniquely determined by $\rho_E(g)=T_{\Gamma_g}$ for all $g\in L\subset\Fr(L)$. Here $U(W)$ is the group of isometric
automorphisms of $W$.

\begin{prop}\label{prop:almost representation associated to an almost flat bundle}
  There exist constants $C,\delta>0$ which depend on $X$, the basepoint $v_0$, the finite presentation $\pi_1(X;v_0)=\langle
  L\mid R\rangle$, and the representing simplicial loops $\Gamma_g$, but not on the $\epsilon$-flat bundle $E\to X$, such that
  $\|\rho_E(r)-\id\|<C\epsilon$ for all $r\in R$ if $\epsilon\leq\delta$.
\end{prop}
\begin{proof}
  Write $r\in R$ as a product $r=g_1\cdots g_n$ of elements $g_k\in L\cup L^{-1}$. Put $\Gamma_k=\Gamma_{g_k}$ if $g_k\in L$,
  and $\Gamma_k=\bar\Gamma_{g_k^{-1}}$ otherwise. Then $\rho_E(r)=T_{\Gamma_1*\cdots*\Gamma_n}$. Since $r$ is contained in the
  kernel of the map $\Fr(L)\to\pi_1(X;v_0)$, the simplicial loop $\Gamma(r)=\Gamma_1*\cdots*\Gamma_n$ is contractible.
  Therefore, Theorem~\ref{thm:transport along contractible loops} implies that $\|T_{\Gamma(r)}-\id\|\leq C_r\cdot\epsilon$ if
  $\epsilon\leq\delta_r$, for some constants $C_r,\delta_r>0$. The claim follows with $\delta=\min_{r\in R}\delta_r$ and
  $C=\max_{r\in R}C_r$.
\end{proof}

This property of $\rho_E$ can be formalized as follows:

\begin{dfn}[{\cite{connes-gromov-moscovici-conjecture-de-novikov}}]
  An \emph{$\epsilon$-almost representation} of a finitely presented group $G=\langle L\mid R\rangle$ on a Hilbert $B$-module
  $W$ is a group homomorphism $\rho\colon\Fr(L)\to U(W)$ such that $\|\rho(r)-\id\|\leq\epsilon$ for all $r\in R$.
\end{dfn}

Therefore, Proposition~\ref{prop:almost representation associated to an almost flat bundle} can be reformulated by saying that
$\rho_E\colon\Fr(L)\to U(W)$ is a $C\epsilon$-representation if $E\to X$ is an $\epsilon$-flat bundle with $\epsilon\leq\delta$.
We will mainly be interested in almost flat bundles and almost representations in the limit $\epsilon\to 0$, which can be
formalized as follows.

\begin{dfn}
  An \emph{asymptotic representation} \cite{manuilov-mishchenko-almost-asymptotic-fredholm-representations}
  of a finitely presented group $G=\langle L\mid R\rangle$ over the C*-algebra $B$ is a sequence $(W_n,\rho_n)_{n\in\N}$ of
  $\epsilon_n$-representations $\rho_n\colon\Fr(L)\to U(W_n)$ with $\epsilon_n\to 0$, where the $W_n$ are all Hilbert $B$-modules.

  Similarly, an \emph{asymptotically flat Hilbert $B$-module bundle} over a simplicial complex $X$ is a sequence
  $(E_n)_{n\in\N}$ of $\epsilon_n$-flat Hilbert $B$-module bundles $E_n\to X$, such that $\epsilon_n\to 0$.
\end{dfn}

Now Proposition~\ref{prop:almost representation associated to an almost flat bundle} implies that the almost representations
associated to an asymptotically flat Hilbert $B$-module bundle $(E_n)_{n\in\N}$ form an asymptotic representation
$(W_n,\rho_{E_n})_{n\in\N}$ of the fundamental group of $X$. Of course, this asymptotic representation depends on the choice of
generating set $L$ and on the representing curves $\Gamma_g$. However, these choices lead to equivalent asymptotic
representations in a sense that we will describe next.

\begin{lem}\label{lem:equivalence of asymptotic representations}
  Let $G$ be a group with two finite presentations $G=\langle L_1\mid R_1\rangle$ and $G=\langle L_2\mid R_2\rangle$. For
  $k=1,2$ we denote by $\pi_k\colon\Fr(L_k)\to G$ the canonical projections. For $k=1,2$ and $n\in\N$ let
  $\rho_{k,n}\colon\Fr(L_k)\to U(W_n)$ be almost representations such that $(W_n,\rho_{k,n})_{n\in\N}$ are asymptotic
  representations for $k=1,2$. Then the following are equivalent:
  \begin{enumerate}
    \item There exist set-theoretic sections $s_k\colon G\to\Fr(L_k)$ of the projections $\pi_k$ such that
      \begin{equation}
        \lim_{n\to\infty}\|\rho_{1,n}(s_1(g))-\rho_{2,n}(s_2(g))\|=0
        \label{eq:equivalence of asymptotic representations}
      \end{equation}
      for all $g\in G$.
    \item Equation \eqref{eq:equivalence of asymptotic representations} holds for all pairs of set-theoretic sections $s_k\colon
      G\to\Fr(L_k)$ of $\pi_k$.
  \end{enumerate}
\end{lem}
\begin{proof}
  The claim follows straightforwardly from the observation that
  \[
    \lim_{n\to\infty}\|\rho_{k,n}(r)-\id_{W_n}\|=0
  \]
  for all $r\in\ker\pi_k$, which in turn follows from the fact that $r$ can be written as a product of conjugates of elements in
  $R\cup R^{-1}$.
\end{proof}

If two asymptotic representations $(\rho_{1,n})_{n\in\N}$ and $(\rho_{2,n})_{n\in\N}$ satisfy the two equivalent conditions of
Lemma~\ref{lem:equivalence of asymptotic representations} then they are called \emph{asymptotically equivalent}. It is clear
from the definition that asymptotic equivalence is an equivalence relation.

\begin{prop}\label{prop:associated asymptotic representations for different presentations}
  Let $X$ be a connected simplicial complex, let $v_0\in X$ be a base vertex, and suppose that $G=\pi_1(X;v_0)$ has two finite
  presentations $G=\langle L\mid R\rangle$ and $G'=\langle L'\mid R'\rangle$. For each $g\in L$ let $\Gamma_g$ be a simplicial
  loop representing $g$, and for each $g'\in L'$ let $\Gamma_{g'}'$ be a simplicial loop representing $g'$.

  Consider an asymptotically flat Hilbert $B$-module bundle $(E_n)_{n\in\N}$, and let $\rho_{E_n}\colon\Fr(L)\to U(W_n)$ and
  $\rho_{E_n}'\colon\Fr(L')\to U(W_n)$ be the associated almost representations. Then the asymptotic representations
  $(\rho_{E_n})_{n\in\N}$ and $(\rho_{E_n}')_{n\in\N}$ are asymptotically equivalent.
\end{prop}
\begin{proof}
  Choose sections $s\colon G\to\Fr(L)$ and $s'\colon G\to\Fr(L')$, and consider $g\in G$. Then there exist simplicial loops
  $\Gamma$ and $\Gamma'$ in $X$, both representing $g$, such that $\rho_{E_n}(s(g))=T_\Gamma^n$ and
  $\rho_{E_n'}(s'(g))=T_{\Gamma'}^n$ for all $n\in\N$, where $T_\Gamma^n$ and $T_{\Gamma'}^n$ denote the transport operators
  along $\Gamma$ and $\Gamma'$ in $E_n$, respectively. In particular, $\bar\Gamma'*\Gamma$ is a contractible simplicial loop in
  $X$. Thus,
  \[
    \lim_{n\to\infty}\|\rho_{E_n}(s(g))-\rho_{E_n'}(s'(g))\|=\lim_{n\to\infty}\|T_{\bar\Gamma'*\Gamma}^n-\id\|=0
  \]
  by Theorem~\ref{thm:transport along contractible loops}.
\end{proof}

Now let us introduce the datum of a Fredholm operator into the picture of asymptotic representations.

\begin{dfn}
  Let $B$ be a unital C*-algebra. An \emph{$\epsilon$-Fredholm representation} of a finitely presented group $G=\langle L\mid
  R\rangle$ is an $\epsilon$-representation $\rho\colon\Fr(L)\to U(W)$ by even operators on a graded Hilbert $B$-module $W$,
  together with an odd operator $F\in\mathcal L_B(W)$ such that $F^2-\id$, $F^*-F$, $[\rho(g),F]$, and $[\rho(g),F^*]$ are
  compact operators for all $g\in L$.

  An \emph{asymptotic Fredholm representation} of $G=\langle L\mid R\rangle$ is a sequence of $\epsilon_n$-Fredholm
  representations $(W_n,\rho_n,F_n)_{n\in\N}$, all of which have the same underlying unital C*-algebra $B$, such that
  $\epsilon_n\to 0$.

  Similarly, an \emph{asymptotically flat Fredholm bundle} is a sequence $(E_n,F_n)_{n\in\N}$ of $\epsilon_n$-flat Fredholm
  bundles such that $\epsilon_n\to 0$.
\end{dfn}

Not surprisingly, we can associate to an almost flat Fredholm bundle over a finite connected simplicial complex $X$ an almost
Fredholm representation of the fundamental group of $X$. Indeed, if $(E,F_E)$ is an $\epsilon$-flat Fredholm bundle, then we put
$\hat F_E=F_{v_0}(v_0)\in\mathcal L_B(W)$, which is of course just the operator $F$ in the definition of the projection $\tilde
P^E$ which appeared in Theorem~\ref{thm:alternative calculation of the index of an almost flat Fredholm bundle}.

\begin{lem}
  If $\Gamma$ is an arbitrary simplicial path in $X$, then $[T_\Gamma,\hat F_E]\in\mathcal K_B(W)$.
\end{lem}
\begin{proof}
  It suffices to prove this in the case where $\Gamma=(v,v')$. Now the statement follows immediately from the fact that $\hat
  F-F_v(\frac 12(v+v'))$ and $\hat F-F_{v'}(\frac 12(v+v'))$ are compact by definition of an almost flat Fredholm bundle.
\end{proof}

Thus, $(W,\rho_E,\hat F_E)$ is a $C\epsilon$-Fredholm representation of $\pi_1(X;v_0)$ if $\epsilon$ is sufficiently small. In
particular, $(W_n,\rho_n,\hat F_n)_{n\in\N}$ is an asymptotic Fredholm representation of $\pi_1(X;v_0)$ if $(E_n,F_n)_{n\in\N}$
is an asymptotically flat Fredholm bundle, where $W_n$ denotes the fiber of $E_n$ and where $\rho_n$ is the almost representation
associated to $E_n$.

\section{D-theory and E-theory}

Later we will associate an asymptotic index to an asymptotically flat Fredholm representation. This asymptotic index takes its
value in the \emph{Thomsen D-theory group} $D(\Sigma C^*G,B)$ \cite{thomsen-discrete-asymptotic-homomorphisms}.
We will recall the definition and a few basic properties of D-theory groups and the related E-theory groups of Connes and Higson
in this section.

Let $B$ be a C*-algebra, which we first assume to be separable. The \emph{discrete asymptotic algebra} over $B$ is the
C*-algebra
\[
  \mathcal A_\delta B=\frac{C_b(\N,B)}{C_0(\N,B)}
\]
where $C_b(\N,B)$ is the C*-algebra of bounded sequences in $B$, and $C_0(\N,B)$ is the ideal of those sequences which tend to
zero at infinity. Obviously, $\mathcal A_\delta$ is an endofunctor on the category of C*-algebras. A \emph{discrete asymptotic
homomorphism} from $A$ to $B$ is a *-ho\-mo\-mor\-phism $A\to\mathcal A_\delta B$.  Two such discrete asymptotic homomorphisms
$f,g\colon A\to\mathcal A_\delta B$ are called \emph{asymptotically homotopic} if there is an asymptotic homomorphism $H\colon
A\to\mathcal A_\delta IB$ with $f=\mathcal A_\delta\ev_0\circ H$ and $g=\mathcal A_\delta\ev_1\circ H$. Here $IB=C([0,1],B)$ is
the C*-algebra of continuous $B$-valued functions on the unit interval, and $\ev_\tau\colon IB\to B$ is the map given by
evaluation at $\tau\in[0,1]$.

We denote by $\bbrack{A,B}_\delta$ the set of asymptotic homotopy classes of discrete asymptotic homomorphisms, and put
\[
  D(A,B)=\bbrack{\Sigma A\otimes\mathcal K,\Sigma^2B\otimes\mathcal K}_\delta,
\]
where $\Sigma A=C_0(\R;A)\cong C_0(\R)\otimes A$ is the \emph{suspension} of $A$ and $\mathcal K$ is the C*-algebra of compact
operators on a separable Hilbert space.

Similarly, we recall the definition Connes's and Higson's E-theory groups
\cite{connes-higson-deformations-morphismes-asymptotiques} as described in \cite{guentner-higson-trout-equivariant-e-theory}:
The \emph{asymptotic algebra} over $B$ is the C*-algebra
\[
  \mathcal AB=\frac{C_b(P,B)}{C_0(P,B)},
\]
where $C_b(P,B)$ is the C*-algebra of bounded continuous $B$-valued functions on $P=[0,\infty)\subset\R$, and $C_0(P,B)$ is the
ideal of those continuous functions which vanish at infinity. As before, an \emph{asymptotic homomorphism} from $A$ to $B$ is a
*-ho\-mo\-mor\-phism $A\to\mathcal AB$, and also asymptotic homotopy is defined as before. We denote by $\bbrack{A,B}$ the set
of asymptotic homotopy classes of asymptotic homomorphisms from $A$ to $B$, and put
\[
  E(A,B)=\bbrack{\Sigma A\otimes\mathcal K,\Sigma B\otimes\mathcal K}.
\]
Note that there is a double suspension in front of $B$ in the definition of $D(A,B)$ but only a single suspension in the
definition of $E(A,B)$. The advantage of this definition is that one now gets associative products $E(A,B)\times E(B,C)\to
E(A,C)$, $E(A,B)\times D(B,C)\to D(A,C)$, $D(A,B)\times E(B,C)\to D(A,C)$, and $D(A,B)\times D(B,C)\to D(A,C)$
\cite{thomsen-discrete-asymptotic-homomorphisms}, at least if all of the appearing C*-algebras are separable. All the
composition products are written as $(f,g)\mapsto g\bullet f$: For example, if $f\in D(A,B)$ and $g\in E(B,C)$ then we write
$g\bullet f\in D(A,C)$ for their product. In addition, there are natural maps $D(A,B)\to E(A,B)$ which are compatible with the
products.  Both the E-theory groups and the D-theory groups carry natural abelian group structures such that the various
products are bilinear.

It should be noted that the construction of the products above require that the occurring C*-algebras are separable. However,
the definition of the E-theory groups and the D-theory groups and all the products can be extended to non-separable
C*-algebras by considering $\bbbrack{A,B}=\lim_{A'}\bbrack{A'B}$ where the limit is taken over the partially ordered set of all
separable C*-subalgebras $A'\subset A$, and by putting $E(A,B)=\bbbrack{\Sigma A\otimes\mathcal K,\Sigma B\otimes\mathcal K}$ in
this case. Of course, the two definitions agree for separable C*-algebras, and most of the properties (in particular, all
properties that we will need) of the E-theory groups continue to hold in this setting. Analogously, we define
$\bbbrack{A,B}_\delta=\lim_{A'}\bbrack{A',B}_\delta$ and define $D(A,B)$ in terms of $\bbbrack{-,-}_\delta$ for non-separable
C*-algebras. For details about this construction we refer to the author's doctoral dissertation
\cite[Section~3.7]{hunger-asymptotically-flat-fredholm-bundles}.

For a C*-algebra $B$ there is a natural map $\kappa_B\colon B\to\mathcal AB$ given by $\kappa_B(b)=[t\mapsto b]$. In
particular, if $f\colon A\to B$ is an arbitrary *-ho\-mo\-mor\-phism then there is an associated element $\kappa(f)\in E(A,B)$
given by
\[
  \kappa(f)=\left[\kappa_{\Sigma B\otimes\mathcal K}\circ(\Sigma f\otimes\id_{\mathcal K})\right].
\]
If in addition $g\colon\Sigma B\otimes\mathcal K\to\mathcal A(\Sigma C\otimes\mathcal K)$ is an asymptotic homomorphism, then
also $g\circ(\Sigma f\otimes\id_{\mathcal K})\colon\Sigma A\otimes\mathcal K\to\mathcal A(\Sigma C\otimes\mathcal K)$ is an
asymptotic homomorphism which represents an element in $E(A,C)$, and one can easily show, using the definition of the E-theory
product, that
\begin{equation}
  g\bullet\kappa(f)=\left[g\circ(\Sigma f\otimes\id_{\mathcal K})\right]
  \label{eq:postcomposition with an asymptotic homomorphism}
\end{equation}
in $E(A,C)$. Similarly, \eqref{eq:postcomposition with an asymptotic homomorphism} holds in $D(A,C)$ if $g\colon\Sigma
B\otimes\mathcal K\to\mathcal A_\delta(\Sigma^2 C\otimes\mathcal K)$ is a discrete asymptotic homomorphism.

Essentially by definition, D-theory and E-theory are \emph{stable} in the following sense: Let $P\in\mathcal K$ be any
rank-one projection, and consider the map $f\colon B\to B\otimes\mathcal K$ given by $f(b)=b\otimes P$. Then $f$ induces
isomorphisms $E(A,B)\to E(A,B\otimes\mathcal K)$, $E(B\otimes\mathcal K,C)\to E(B,C)$ and $D(A,B)\to D(A,B\otimes\mathcal K)$,
$D(B\otimes\mathcal K,C)\to D(B,C)$.

Both E-theory and D-theory allow the following tensor product construction: Let $A,B,C$ be arbitrary C*-algebras. Then there is
a natural group homomorphism $-\otimes\id_C\colon E(A,B)\to E(A\otimes C,B\otimes C)$ where $\otimes$ denotes the
\emph{maximal} tensor product. This group homomorphism is defined as follows: Represent an element of the domain $E(A,B)$ by an
asymptotic homomorphism $f\colon\Sigma A\otimes\mathcal K\to\mathcal A(\Sigma B\otimes\mathcal K)$. Then
$[f]\otimes\id_C\in E(A\otimes C,B\otimes C)$ is the class which is represented by the asymptotic homomorphism
\[
  f\otimes\id_C\colon\Sigma A\otimes C\otimes\mathcal K\cong\Sigma A\otimes\mathcal K\otimes C\to\mathcal A(\Sigma
  B\otimes\mathcal K)\otimes C\to\mathcal A(\Sigma B\otimes C\otimes\mathcal K).
\]
In an analogous fashion one can also define $-\otimes\id_C\colon D(A,B)\to D(A\otimes C,B\otimes C)$. Of particular importance
are the \emph{suspension maps} $\Sigma\colon E(A,B)\to E(\Sigma A,\Sigma B)$ and $\Sigma\colon D(A,B)\to D(\Sigma A,\Sigma B)$,
which are isomorphisms by \cite[Proposition~6.17]{guentner-higson-trout-equivariant-e-theory} and
\cite[Theorem~4.2]{thomsen-discrete-asymptotic-homomorphisms}, respectively.

There is another important fact about E-theory which we will need in the following section.

\begin{prop}[{\cite[Proposition 5.1]{connes-higson-asymptotic-homomorphisms}}]
  \label{prop:E-theory element associated to a split short exact sequence}
  Let
  \[
    \begin{tikzpicture}
      \matrix (m) [matrix of math nodes, row sep=3em, column sep=3em, commutative diagrams/every cell]
      {
        0 & J & A & B & 0 \\
      };
      \path[-stealth, commutative diagrams/.cd, every label] (m-1-1) edge (m-1-2) (m-1-2) edge node [above] {$\iota$} (m-1-3)
        (m-1-3) edge node [above] {$\pi$} (m-1-4) (m-1-4) edge (m-1-5);
    \end{tikzpicture}
  \]
  be a short exact sequence of C*-algebras, and let $s\colon B\to A$ be a splitting, i.\,e.\ a *-ho\-mo\-mor\-phism such that
  $\pi\circ s=\id_B$. Then there exists an element $\sigma\in E(A,J)$ such that $\sigma\bullet\kappa(\iota)=\kappa(\id_J)$ and
  $\kappa(\iota)\bullet\sigma+\kappa(s\circ\pi)=\kappa(\id_A)$.\qed
\end{prop}

Finally, we will need a few concrete calculations of E-theory and D-theory groups, namely of the groups $E(\C,B)$ and
$D(\Sigma,B)$ where $\Sigma=\Sigma\C=C_0(\R)$ is the suspension algebra.

For E-theory, the relevant calculation goes back to \cite[Theorem~4.1]{rosenberg-role-of-ktheory} which far predates the
development of E-theory. Let us assume for simplicity that $B$ is unital. Let $p\in M_n(B)$ be a projection. Then we can
consider the *-ho\-mo\-mor\-phism $f_p\colon\C\to B\otimes\mathcal K$ which is determined by the property that $f_p(1)=p$. Now
we define a map
\[
  \Phi_B\colon K_0(B)\to E(\C,B)
\]
by $\Phi_B[p]=\kappa(f_P)\in E(\C,B\otimes\mathcal K)\cong E(\C,B)$. For the proof of the following statement we refer to
\cite[Proposition~2.19]{guentner-higson-trout-equivariant-e-theory}.

\begin{thm}\label{thm:calculation of E(C,B)}
  The map $\Phi_B$ is a group isomorphism for all unital C*-algebras $B$.\qed
\end{thm}

For D-theory, there is a similar isomorphism (although the proof is more complicated), which was already hinted at in Thomsen's
paper \cite{thomsen-discrete-asymptotic-homomorphisms}. Namely, if $B$ is unital then we define a map
\[
  \Psi_B\colon\prod_{n\in\N}K_0(B)\to D(\Sigma,B)
\]
by the prescription $\Psi_B(([p_n])_{n\in\N})=[\Sigma^2f_{(p_n)}\otimes\id_{\mathcal K}]\in\bbrack{\Sigma^2\C\otimes\mathcal
K,\Sigma^2B\otimes\mathcal K\otimes\mathcal K}_\delta=D(\Sigma,B\otimes\mathcal K)\cong D(\Sigma,B)$, where
$f_{(p_n)}\colon\C\to\mathcal A_\delta(B\otimes\mathcal K)$ is the unique discrete asymptotic homomorphism such that
$f_{(p_n)}(1)=[n\mapsto p_n]$. Thus, $\Sigma^2f_{(p_n)}\otimes\id_{\mathcal K}(\phi\otimes\psi\otimes
T)=[n\mapsto\phi\otimes\psi\otimes p_n\otimes T]$ for all $\phi,\psi\in\Sigma$ and $T\in\mathcal K$.

\begin{thm}[{\cite[Theorem~3.8.11]{hunger-asymptotically-flat-fredholm-bundles}}]\label{thm:calculation of D(Sigma,B)}
  The map $\Psi_B$ is a surjective group homomorphism with $\ker\Psi_B=\bigoplus_{n\in\N}K_0(B)$, for all unital C*-algebras
  $B$.\qed
\end{thm}

\section{The asymptotic index of an asymptotic representation}

In this section, we fix a unital C*-algebra $B$ and an asymptotic Fredholm representation $(W_n,\rho_n,F_n)_{n\in\N}$. We want
to construct the asymptotic index
\[
  \asind\left((W_n,\rho_n,F_n)_{n\in\N}\right)\in D(\Sigma C^*G,B),
\]
where $C^*G$ is the \emph{maximal} group C*-algebra of $G$. The construction parallels the calculation of $\ind F_E$ for an
almost flat Fredholm bundle $(E,F_E)$ in Theorem~\ref{thm:alternative calculation of the index of an almost flat Fredholm bundle}.
Firstly, we want to get rid of the Hilbert $B$-modules $W_n$. In order to do this, we choose even isometric isomorphisms
$U_n\colon W_n\oplus\mathcal H_B\to\mathcal H_B$, which exist by Kasparov's Stabilization Theorem. For all $w\in\Fr(L)$ we
define
\[
  \rho_n'(w)=U_n(\rho_n(w)\oplus 0)U_n^*
\]
and
\[
  F_n'=U_n(F_n\oplus\id)U_n^*.
\]
Further, we define
\[
  \rho_n''(w)=U(F_n')((\rho_n'(w)\oplus 0)\oplus 0)U(F_n')^*\in\mathcal L_B(\mathcal H_B')
\]
and
\[
  \tilde\rho_n(w)=V\rho_n''(w)V^*\in\mathcal L_B(\mathcal H_B)
\]
where $V\colon\mathcal H_B'\to\mathcal H_B$ is another even unitary isomorphism. Proposition~\ref{prop:properties of U(F)}
easily implies that $[\rho_n''(g),T]$ is compact for all $g\in L$, where
$T=\begin{psmallmatrix}0&1\\1&0\end{psmallmatrix}\in\mathcal L_B(\mathcal H_B')$. As before, we put $T'=VTV^*$. Then
\[
  \tilde\rho_n(w)\in Q=Q_{T'}=\left\{x\in\mathcal L_B^{\mathrm{ev}}(\mathcal H_B):[x,T']\in\mathcal K_B(\mathcal H_B)\right\}
\]
for all $w\in\Fr(L)$.

\begin{lem}\label{lem:definition of rho for the asymptotic index}
  There exists a unique *-ho\-mo\-mor\-phism $\rho\colon C^*G\to\mathcal A_\delta Q$ such that
  \[
    \rho(\pi(w))=[n\mapsto\tilde\rho_n(w)]
  \]
  for all $w\in\Fr(L)$, where $\pi\colon\Fr(L)\to G$ is the canonical projection and where we identify $\pi(w)\in G$ with its
  image in $C^*G$.
\end{lem}
\begin{proof}
  Uniqueness is clear since $C^*G$ is generated by the elements of $G\subset C^*G$. For existence, we consider
  $P_n=\tilde\rho_n(1)\in Q$. It follows from the definition of $\tilde\rho_n$ and the fact that each $\rho_n$ is a group
  homomorphism that $\tilde\rho_n(ww')=\tilde\rho_n(w)\tilde\rho_n(w')$ for all $w,w'\in\Fr(L)$, and that
  $\tilde\rho_n(w)^*=\tilde\rho_n(w^{-1})$ since the analogous statement is true for $\rho_n$. In particular,
  $\tilde\rho_n(w)^*\tilde\rho_n(w)=\tilde\rho_n(1)=P_n=\tilde\rho_n(w)\tilde\rho_n(w)^*$, so that each $\tilde\rho_n(w)$ is
  unitary in the unital C*-algebra $P_nQP_n$. Thus, each $\tilde\rho_n\colon\Fr(L)\to P_nQP_n$ is a unitary representation, so
  that $\tilde\rho\colon\Fr(L)\to P(\mathcal A_\delta Q)P$, $\tilde\rho(w)=[n\mapsto\tilde\rho_n(w)]$ is a unitary
  representation as well if we define $P=[n\mapsto P_n]\in\mathcal A_\delta Q$.

  Now if $r\in\ker\pi$ is arbitrary, we get that $\tilde\rho(r)=[n\mapsto\tilde\rho_n(r)]=[n\mapsto P_n]=\tilde\rho(1)$ because
  $\lim_{n\to\infty}\|\tilde\rho_n(r)-P_n\|=0$ by definition of an asymptotic representation. Therefore,
  $\tilde\rho\colon\Fr(L)\to P(\mathcal A_\delta Q)P$ descends to a unitary representation $G\to P(\mathcal A_\delta Q)P$, which
  extends to a *-ho\-mo\-mor\-phism $\rho\colon C^*G\to P(\mathcal A_\delta Q)P\subset\mathcal A_\delta Q$ by the universal
  property of the maximal group C*-algebra. It is clear that $\rho(\pi(w))=[n\mapsto\tilde\rho_n(w)]$ as required.
\end{proof}

Now we put
\[
  \hat\rho=\Sigma^2\rho\otimes\id_{\mathcal K}\colon\Sigma^2C^*G\otimes\mathcal K\to\Sigma^2\mathcal A_\delta Q\otimes\mathcal
  K\to\mathcal A_\delta(\Sigma^2Q\otimes\mathcal K).
\]
Therefore, $\hat\rho$ represents an element $[\hat\rho]\in\bbrack{\Sigma^2C^*G\otimes\mathcal K,\Sigma^2Q\otimes\mathcal
K}_\delta=D(\Sigma C^*G,Q)$. We consider the split short exact sequence
\[
  \begin{tikzpicture}
    \matrix (m) [matrix of math nodes, row sep=3em, column sep=3em, commutative diagrams/every cell]
    {
      0 & \mathcal K_B(H_B) & Q & \mathcal L_B(H_B) & 0. \\
    };
    \path[-stealth, commutative diagrams/.cd, every label] (m-1-1) edge (m-1-2) (m-1-2) edge node [above] {$i_{T'}$} (m-1-3)
      (m-1-3) edge node [above] {$\pi_{T'}$} (m-1-4) (m-1-4) edge [bend left=30] node [below] {$s_{T'}$} (m-1-3) edge (m-1-5);
  \end{tikzpicture}
\]
By Proposition~\ref{prop:E-theory element associated to a split short exact sequence} there exists a class $\sigma\in
E(Q,\mathcal K_B(H_B))$ such that $\kappa(i_{T'})\bullet\sigma+\kappa(s_{T'}\pi_{T'})=\kappa(\id_Q)$ and
$\sigma\bullet\kappa(i_{T'})=\kappa(\id_{\mathcal K_B(H_B)})$.

\begin{dfn}
  The \emph{asymptotic index} of the asymptotic Fredholm representation is defined to be
  \[
    \asind\left((W_n,\rho_n,F_n)_{n\in\N}\right)=\sigma\bullet[\hat\rho]\in D(\Sigma C^*G,\mathcal K_B(H_B))\cong D(\Sigma
    C^*G,B).
  \]
\end{dfn}

It is straightforward to show that $\asind((W_n,\rho_n,F_n)_{n\in\N})$ is independent of the choices of isomorphisms $U_n$ and
$V$, and that asymptotically equivalent asymptotic homomorphisms have the same asymptotic index.

\section{The index of an asymptotically flat Fredholm bundle}

We have gathered all the necessary preliminaries to formulate and prove our main theorem which relates the index of an
asymptotically flat Fredholm bundle to the asymptotic index of the associated asymptotic Fredholm representation. Let $X$ be a
finite connected simplicial complex, let $B$ be a unital C*-algebra, and let $(E_n,F_n)_{n\in\N}$ be an asymptotically flat
Fredholm bundle over $X$ with underlying C*-algebra $B$. Choose a finite presentation $G=\pi_1(X;v_0)=\langle L\mid R\rangle$
and simplicial loops $\Gamma_g$ which represent the generators $g\in L$. Let $(W_n,\rho_n,\hat F_n)_{n\in\N}$ be the asymptotic
Fredholm representation which is associated to the asymptotically flat Fredholm bundle $(E_n,F_n)_{n\in\N}$ as described in
section~\ref{sec:asymptotic Fredholm representations}.

Consider the \emph{Mishchenko bundle}
\[
  M_X=(C^*G\times\tilde X)/G,
\]
where $\tilde X$ is the universal cover of $X$ and where the quotient is taken by the diagonal action of $G$ on the product
$C^*G\times\tilde X$. Then $M_X$ is a Hilbert $C^*G$-module bundle over $X$ and hence defines a class $[M_X]\in K^0(X;C^*G)\cong
K_0(C(X)\otimes C^*G)$. Let $\Phi=\Phi_{C(X)\otimes C^*G}\colon K_0(C(X)\otimes C^*G)\to E(\C,C(X)\otimes C^*G)$ be the
isomorphism of Theorem~\ref{thm:calculation of E(C,B)}.

\begin{thm}\label{thm:main theorem}
  Under the identification $D(\Sigma,C(X)\otimes B)\cong\prod_{n\in\N}K_0(C(X)\otimes B)/\bigoplus_{n\in\N}K_0(C(X)\otimes B)$
  of Theorem~\ref{thm:calculation of D(Sigma,B)}, the classes
  \[
    \left(\id_{C(X)}\otimes\asind\left((W_n,\rho_n,\hat F_n)_{n\in\N}\right)\right)\bullet\Sigma\Phi[M_X]\in
    D(\Sigma,C(X)\otimes B)
  \]
  and
  \[
    \left[(\ind F_n)_{n\in\N}\right]\in\frac{\prod_{n\in\N}K_0(C(X)\otimes B)}{\bigoplus_{n\in\N}K_0(C(X)\otimes B)}
  \]
  coincide.
\end{thm}
\begin{proof}
  We begin with a few simplifications. Denote by $\Phi_{n,v}\colon S_v\times W_n\to E_n|_{S_v}$ the local trivializations of
  $E_n$. Choose a maximal tree $T\subset X$. For every vertex $V\in V_X$ of $X$ let $\Gamma_v$ be the unique simple simplicial
  path connecting the base vertex $v_0$ to $v$ inside $T$, and put $\Phi_{n,v}'(x,\xi)=\Phi_{n,v}(x,T_{\Gamma_v}\xi)$. Then a
  straightforward calculation shows that replacing $\Phi$ by $\Phi'$ does not change the asymptotic Fredholm representation
  $(W_n,\rho_n,\hat F_n)_{n\in\N}$, but all transport operators along paths in $T$ are equal to the identity. Therefore, we may
  assume without loss of generality that parallel transport in $T$ is trivial.

  As a second simplification, note that we may choose a concrete presentation of the fundamental group and of representing loops
  in the definition of $(W_n,\rho_n,\hat F_n)_{n\in\N}$ because these changes leave the asymptotic equivalence class of
  $(W_n,\rho_n,\hat F_n)_{n\in\N}$ invariant by Proposition~\ref{prop:associated asymptotic representations for different
  presentations}, and asymptotically equivalent asymptotic Fredholm representations have the same asymptotic index.

  We will use the following presentation of $G=\pi_1(X;v_0)$: For any two vertices $v,w\in X$ such that $[v,w]$ is an edge
  in $X$ we associate the simplicial loop $\Gamma_{v,w}=\bar\Gamma_w*(v,w)*\Gamma_v$ at $v_0$ and put
  $g_{v,w}=[\Gamma_{v,w}]\in\pi_1(X;v_0)$. Then $L=\{g_{v,w}:v,w\}$ is a finite generating set for $\pi_1(X;v_0)$, and it
  follows from the Seifert--van Kampen Theorem that this generating set admits a finite set $R$ of relations.

  Note that
  \begin{equation}
    T_{\Gamma_{(v,w)}}=T_{\bar\Gamma_w}T_{(v,w)}T_{\Gamma_v}=T_{(v,w)}=\Psi_{w,v}\left(\frac 12(v+w)\right)
    \label{eq:transport operators if transport in a tree is trivial}
  \end{equation}
  if $T$ denotes the transport operator in any of the bundles $E_n$ because we assumed that transport in $T$ is trivial.

  Now let us return to the proof of the theorem. Consider the map $\Psi=\Psi_{C(X)\otimes B}\colon\prod_{n\in\N}K_0(C(X)\otimes
  B)\to D(\Sigma,C(X)\otimes B)$. Then by Theorem~\ref{thm:alternative calculation of the index of an almost flat Fredholm
  bundle} the statement of the theorem can be reformulated as
  \begin{equation}
    \Psi\left(\left(\rho_X[\tilde P^{E_n}]\right)_{n\in\N}\right)=\left(\id_{C(X)}\otimes\asind\left((W_n,\rho_n,\hat
    F_n)_{n\in\N}\right)\right)\bullet\Sigma\Phi[M_X].\label{eq:first reformulation of the main theorem}
  \end{equation}
  Recall that $\asind((W_n,\rho_n,\hat F_n)_{n\in\N})=\sigma\bullet[\hat\rho]$ where $\hat\rho=\Sigma^2\rho\otimes\id_{\mathcal
  K}$ and $\rho\colon C^*G\to\mathcal A_\delta Q$ is such that $\rho(\pi(w))=[n\mapsto\tilde\rho_n(w)]$ for all $w\in\Fr(L)$,
  and where $\sigma\in E(Q,\mathcal K_B(H_B))$ is the class associated to the split short exact sequence
  \[
    \begin{tikzpicture}
      \matrix (m) [matrix of math nodes, row sep=3em, column sep=3em, commutative diagrams/every cell]
      {
        0 & \mathcal K_B(H_B) & Q & \mathcal L_B(H_B) & 0. \\
      };
      \path[-stealth, commutative diagrams/.cd, every label] (m-1-1) edge (m-1-2) (m-1-2) edge (m-1-3) (m-1-3) edge (m-1-4)
        (m-1-4) edge (m-1-5);
    \end{tikzpicture}
  \]
  It is straightforward to show that
  $\id_{C(X)}\otimes(\sigma\bullet[\hat\rho])=(\id_{C(X)}\otimes\sigma)\bullet(\id_{C(X)}\otimes[\hat\rho])$, so that the right
  hand side in~\eqref{eq:first reformulation of the main theorem} equals
  $(\id_{C(X)}\otimes\sigma)\bullet(\id_{C(X)}\otimes[\hat\rho])\bullet\Sigma\Phi[M_X]$. Consider the diagram
  \begingroup\scriptsize
  \[
    \begin{tikzpicture}
      \matrix (m) [matrix of math nodes, row sep=4em, column sep=1.9em, commutative diagrams/every cell]
      {
        0 &[-0.5em] \prod_{n\in\N}K_0(C(X)\otimes\mathcal K_B(H_B)) & \prod_{n\in\N}(C(X)\otimes Q) &
        \prod_{n\in\N} K_0(C(X)\otimes\mathcal L_B(H_B)) &[-0.5em] 0 \\
        0 & D(\Sigma,C(X)\otimes\mathcal K_B(H_B)) & D(\Sigma,C(X)\otimes Q) & D(\Sigma,C(X)\otimes\mathcal L_B(H_B)) & 0 \\
      };
      \path[-stealth, commutative diagrams/.cd, every label]
        (m-1-1) edge (m-1-2) (m-1-2) edge (m-1-3) (m-1-3) edge (m-1-4) (m-1-4) edge (m-1-5) edge [bend left=30] (m-1-3)
        (m-2-1) edge (m-2-2) (m-2-2) edge (m-2-3) (m-2-3) edge (m-2-4) (m-2-4) edge (m-2-5) edge [bend left=30] (m-2-3)
        (m-1-2) edge node [left] {$\Psi$} (m-2-2) (m-1-3) edge node [left] {$\Psi$} (m-2-3) (m-1-4) edge node [right] {$\Psi$}
        (m-2-4);
    \end{tikzpicture}
  \]
  \endgroup
  of split short exact sequences of abelian groups. Since $\Psi$ is natural, the diagram commutes. Let $\rho_X'\colon
  D(\Sigma,C(X)\otimes Q)$ be the group homomorphism associated to the bottom sequence, so that $\rho_X'$ is uniquely determined
  by the equation
  \[
    \kappa(\id_{C(X)}\otimes i_{T'})\bullet\rho_X'(\eta)+\kappa(\id_{C(X)}\otimes s_{T'}\pi_{T'})\bullet\eta=\eta
  \]
  for all $\eta\in D(\Sigma,C(X)\otimes Q)$. However, this equation is fulfilled by the map which is given by postcomposition
  with $\id_{C(X)}\otimes\sigma\in E(C(X)\otimes Q,C(X)\otimes\mathcal K_B(H_B))$. Therefore, we must have
  $\rho_X'(\eta)=(\id_{C(X)}\otimes\sigma)\bullet\eta$ for all $\eta\in D(\Sigma,C(X)\otimes Q)$. Now
  Lemma~\ref{lem:homomorphism associated to a split exact sequence of abelian groups} implies that
  \begin{align*}
    \Psi\left(\left(\rho_X[\tilde P^{E_n}]\right)_{n\in\N}\right)
    &=\Psi\circ\left(\prod_{n\in\N}\rho_X\right)\left(\left([\tilde P^{E_n}]\right)_{n\in\N}\right)\\
    &=\rho_X'\circ\Psi\left(\left([\tilde P^{E_n}]\right)_{n\in\N}\right)\\
    &=\left(\id_{C(X)}\otimes\sigma\right)\bullet\Psi\left(\left([\tilde P^{E_n}]\right)_{n\in\N}\right).
  \end{align*}
  Therefore, in order to prove~\eqref{eq:first reformulation of the main theorem} it suffices to prove that
  \begin{equation}
    \Psi\left(\left([\tilde P^{E_n}]\right)_{n\in\N}\right)=\left(\id_{C(X)}\otimes[\hat\rho]\right)\bullet\Sigma\Phi[M_X]\in
    D(\Sigma,C(X)\otimes Q).\label{eq:second reformulation of the main theorem}
  \end{equation}
  The left hand side in~\eqref{eq:second reformulation of the main theorem} is defined to be the class in
  $\bbrack{\Sigma^2\C\otimes\mathcal K,\Sigma^2C(X)\otimes Q\otimes\mathcal K\otimes\mathcal K}_\delta\cong D(\Sigma,C(X)\otimes
  Q)$ of the discrete asymptotic homomorphism $\Sigma^2g\otimes\id_{\mathcal K}$ where $g\colon\C\to\mathcal
  A_\delta(C(X)\otimes Q\otimes\mathcal K)$ is determined by $g(1)=[n\mapsto\tilde P^{E_n}]$.

  On the other hand, the class $\Sigma\Phi[M_X]\in\bbrack{\Sigma^2\C\otimes\mathcal K,\Sigma^2C(X)\otimes C^*G\otimes\mathcal
  K\otimes\mathcal K}\cong E(\Sigma,\Sigma C(X)\otimes C^*G)$ is given by the class $\kappa(\Sigma^2f\otimes\id_{\mathcal K})$
  where $f\colon\C\to C(X)\otimes C^*G\otimes\mathcal K$ is determined by $f(1)=P^{M_X}$ where $P^{M_X}$ is a projection in a
  matrix algebra over $C(X)\otimes C^*G$ whose image is isomorphic to $M_X$.
  In particular, equation~\eqref{eq:postcomposition with an asymptotic homomorphism} on page~\pageref{eq:postcomposition with an
  asymptotic homomorphism} implies
  that the right hand side in~\eqref{eq:second reformulation of the main theorem} is given by the class in $E(\Sigma,C(X)\otimes
  B\otimes\mathcal K)$ of the discrete asymptotic homomorphism $(\id_{C(X)}\otimes\hat\rho\otimes\id_{\mathcal
  K})\circ(\Sigma^2f\otimes\id_{\mathcal K})=\Sigma^2((\id_{C(X)}\otimes\rho\otimes\id_{\mathcal K})\circ f)\otimes
  \id_{\mathcal K}\colon\Sigma^2\C\otimes\mathcal K\to\mathcal A_\delta(\Sigma^2C(X)\otimes Q\otimes\mathcal K)$.

  Therefore, in order to prove~\eqref{eq:second reformulation of the main theorem} it suffices to show that $g$ and
  $(\id_{C(X)}\otimes\rho\otimes\id_{\mathcal K})\circ f\colon\C\to\mathcal A_\delta(C(X)\otimes Q\otimes\mathcal K)$ are equal,
  or in other words that
  \begin{equation}
    \id_{C(X)}\otimes\rho\otimes\id_{\mathcal K}(P^{M_X})=[n\mapsto\tilde P^{E_n}]\in\mathcal A_\delta(C(X)\otimes
    Q\otimes\mathcal K).\label{eq:third reformulation of the main theorem}
  \end{equation}

  Let $V_X=\{v_1,\ldots,v_n\}$ be the ordering of vertices of $X$ which goes into the definition of $\tilde P^{E_n}$. It is
  easy to see that one can take
  \[
    P^{M_X}(x,-)=\left(\sqrt{\lambda_j(x)\lambda_k(x)}g_{(v_k,v_j)}\right)_{j,k}\in M_n(\mathcal L_B(C^*G)),
  \]
  as the definition of $P^{M_X}$, so that
  \[
    \id_{C(X)}\otimes\rho\otimes\id_{\mathcal K}(P^{M_X})
    =\left[n\mapsto\left(x\mapsto\left(\sqrt{\lambda_j(x)\lambda_k(x)}\tilde\rho_n(g_{(v_k,v_j)})\right)_{j,k}\right)\right].
  \]
  On the other hand, the definition of $\tilde P^{E_n}$ is
  \[
    \tilde P^{E_n}(x,-)=\left(\sqrt{\lambda_j(x)\lambda_k(x)}(V\Psi_{n,jk}''(x)V^*)\right)_{j,k}.
  \]
  Therefore, we have to prove that $\lim_{n\to\infty}\|\tilde\rho_n(g_{(v_k,v_j)})-V\Psi_{n,jk}''(x)V^*\|=0$ uniformly in $x\in
  S_j\cap S_k$ for all $j,k$. However, a straightforward calculation using the definition of $\tilde\rho_n$ and $\Psi_{n,jk}''$
  shows that
  \[
    \|\tilde\rho_n(g_{(v_k,v_j)})-V\Psi_{n,jk}''(x)V^*\|=\|\rho_n(g_{(v_k,v_j)})-\Psi_{n,jk}(x)\|.
  \]
  By the definition of $\rho_n$ and by~\eqref{eq:transport operators if transport in a tree is trivial} we have
  $\rho_n(g_{(v_k,v_j)})=\Psi_{n,jk}(\frac 12(v_j+v_k))$. Since $E_n$ is $\epsilon_n$-flat, we have
  \[
    \|\tilde\rho_n(g_{(v_k,v_j)})-V\Psi_{n,jk}''(x)V^*\|\leq\epsilon_n\to 0.
  \]
  This completes the proof of~\eqref{eq:third reformulation of the main theorem}.
\end{proof}

\section{The Strong Novikov Conjecture}

In the last two sections of this paper, we will give two applications of Theorem~\ref{thm:main theorem}. In this section, we
will show how to use the theorem in order to prove special cases of the Strong Novikov Conjecture. We fix a finite connected
simplicial complex $X$.

One can define the \emph{K-homology groups} of a compact space $X$ by $K_{j}(X)=E(C(X),\Sigma^j\C)$ for $j\in\N$. Bott
periodicity implies that $K_j(X)\cong K_{j+2}(X)$ for all $j$, so that one actually only needs to consider $K_0(X)=E(C(X),\C)$
and $K_1(X)=E(C(X),\Sigma)$.

Recall from \cite[Section~4]{higson-bivariant-k-theory-and-the-novikov-conjecture} that the \emph{(analytic) assembly map}
$\mu_X\colon K_*(X)\to K_*(C^*\pi_1(X;v_0))$ for the complex $X$ is defined by the equation
\[
  \Phi(\mu_X(\eta))=(\id_{C^*\pi_1(X;v_0)}\otimes\eta)\bullet\Phi[M_X]
\]
where the maps $\Phi$ are the natural isomorphisms from Theorem~\ref{thm:calculation of E(C,B)} and where $[M_X]\in
K_0(C^*\pi_1(X;v_0)\otimes C(X))$ is the class of the Mishchenko bundle.\footnote{The original definition of the analytic
  assembly map, in terms of KK-theory, is due to Kasparov
  \cite[Definition~9.2]{kasparov-k-theory-group-C*-algebras-higher-signatures-conspectus}. However, the two definitions agree by
\cite[Section~4]{higson-bivariant-k-theory-and-the-novikov-conjecture}.} The \emph{Strong Novikov Conjecture} states that
$\mu_X\otimes\Q\colon K_*(X)\otimes\Q\to K_*(C^*\pi_1(X;v_0))\otimes\Q$ is injective whenever $X\simeq B\pi_1(X;v_0)$ is a
classifying space for its fundamental group. The main theorem of this section will show how one can prove the Strong Novikov
Conjecture in the presence of sufficiently many almost flat Fredholm bundles.

Recall that the \emph{Kronecker pairing} relating K-homology and K-theory is defined as follows: If $\eta\in K_j(X)$ and $\xi\in
K^0(X;B)=K_0(C(X)\otimes B)$ are arbitrary then we define $\langle\eta,\xi\rangle\in K_j(B)$ by
\[
  \langle\eta,\xi\rangle=(\eta\otimes\id_B)\bullet\Phi(\xi)\in E(\C,\Sigma^jB)\cong K_j(B).
\]

\begin{prop}\label{prop:calculation of asymptotic Kronecker pairings}
  Let $X$ be a finite connected simplicial complex, let $B$ be a unital C*-algebra, and let $(E_n,F_n)_{n\in\N}$ be an
  asymptotically flat Fredholm bundle over $X$, with underlying C*-algebra $B$. Let $(W_n,\rho_n,\hat F_n)_{n\in\N}$ be the
  assoiated asymptotic Fredholm representation.

  If $\eta\in K_0(X)=E(C(X),\C)$ then
  \[
    \Psi\left((\langle\eta,\ind F_n\rangle)_{n\in\N}\right)=\asind\left((W_n,\rho_n,\hat
    F_n)_{n\in\N}\right)\bullet\Sigma\Phi(\mu_X(\eta))\in D(\Sigma,B)
  \]
  where $\Psi$ is as in Theorem~\ref{thm:calculation of D(Sigma,B)}. Similarly, if $\eta\in K_1(X)=E(C(X),\Sigma)$ then
  \[
    \Psi\left((\langle\eta,\ind F_n\rangle)_{n\in\N}\right)=\Sigma\left(\asind\left((W_n,\rho_n,\hat
    F_n)_{n\in\N}\right)\right)\bullet\Sigma\Phi(\mu_X(\eta))\in D(\Sigma,\Sigma B).
  \]
\end{prop}
\begin{proof}
  In the case $\eta\in K_0(X)$ we consider the commuting diagram
  \begingroup\small
  \[
    \begin{tikzpicture}
      \matrix (m) [matrix of math nodes, row sep=3em, column sep=3em, commutative diagrams/every cell]
      {
        E(C(X),\C) &[2.5em] E(C^*G\otimes C(X),C^*G) & E(\C,C^*G) \\
         & E(\Sigma C^*G\otimes C(X),\Sigma C^*G) & E(\Sigma,\Sigma C^*G) \\
         E(B\otimes C(X),B) & D(\Sigma C^*G\otimes C(X),B) & D(\Sigma,B) \\
      };
      \path[-stealth, commutative diagrams/.cd, every label] (m-1-1) edge (m-1-2) edge (m-3-1)
        (m-1-2) edge node [above] {$\Phi[M_X]$} (m-1-3) edge (m-2-2) (m-1-3) edge (m-2-3)
        (m-2-2) edge node [above] {$\Sigma\Phi[M_X]$} (m-2-3) edge node [left] {$\asind$} (m-3-2)
        (m-2-3) edge node [left] {$\asind$} (m-3-3)
        (m-3-1) edge node [above] {$\asind\otimes\id_{C(X)}$} (m-3-2)
        (m-3-2) edge node [above] {$\Sigma\Phi[M_X]$} (m-3-3);
    \end{tikzpicture}
  \]
  \endgroup
  where the unlabeled arrows are tensor products with the corresponding identities, and the labeled arrows are composition
  products with the respective elements. Of course, here $\asind=\asind((W_n,\rho_n,\hat F_n)_{n\in\N})$ is the asymptotic index
  of the asymptotic Fredholm representation associated to the asymptotically flat Fredholm bundle $(E_n,F_n)_{n\in\N}$.

  The assembly map is the composition along the top row, under the identifications $\Phi\colon K_0(X)\to
  E(C(X),\C)$ and $\Phi\colon K_0(C^*G)\to E(\C,C^*G)$. By associativity of the composition product,
  the composition along the bottom row is given by precomposition with the element
  \[
    \left(\asind\left((W_n,\rho_n,\hat F_n)_{n\in\N}\right)\otimes\id_{C(X)})\right)\bullet\Sigma\Phi[M_X]
    \in D(\Sigma,B\otimes C(X)),
  \]
  which equals $\Psi[\left(\ind F_n\right)_{n\in\N}]$ by Theorem~\ref{thm:main theorem}.
  Now one can show \cite[Corollary~3.9.4]{hunger-asymptotically-flat-fredholm-bundles} that $\eta\in E(C(X),\C)$ is mapped to
  \[
    \left(\id_B\otimes\eta\right)\bullet\Psi\left[(\ind F_n)_{n\in\N}\right]
    =\Psi\left[\left(\langle\eta,\ind F_n\rangle\right)_{n\in\N}\right]\in D(\Sigma,B)
  \]
  under the composition along the left and bottom arrows. Thus, commutativity of the diagram completes the proof in the case
  $\eta\in K_0(X)$.

  In the case where $\eta\in K_1(X)=E(C(X),\Sigma)$, one can carry out the same argument with the diagram
  \begingroup\small
  \[
    \begin{tikzpicture}
      \matrix (m) [matrix of math nodes, row sep=3em, column sep=3em, commutative diagrams/every cell]
      {
        E(C(X),\Sigma) &[2.5em] E(C^*G\otimes C(X),\Sigma C^*G) & E(\C,\Sigma C^*G) \\
         & E(\Sigma C^*G\otimes C(X),\Sigma^2C^*G) & E(\Sigma,\Sigma^2C^*G) \\
         E(B\otimes C(X),\Sigma B) & D(\Sigma C^*G\otimes C(X),\Sigma B) & D(\Sigma,\Sigma B) \\
      };
      \path[-stealth, commutative diagrams/.cd, every label] (m-1-1) edge (m-1-2) edge (m-3-1)
        (m-1-2) edge node [above] {$\Phi[M_X]$} (m-1-3) edge (m-2-2) (m-1-3) edge (m-2-3)
        (m-2-2) edge node [above] {$\Sigma\Phi[M_X]$} (m-2-3) edge node [left] {$\Sigma(\asind)$} (m-3-2)
        (m-2-3) edge node [left] {$\Sigma(\asind)$} (m-3-3)
        (m-3-1) edge node [above] {$\asind\otimes\kappa(\id_{C(X)})$} (m-3-2)
        (m-3-2) edge node [above] {$\Sigma\Phi[M_X]$} (m-3-3);
    \end{tikzpicture}
  \]
  \endgroup
  instead of the diagram for $K_0(X)$.
\end{proof}

Our application to the Strong Novikov Conjecture is the following.

\begin{thm}\label{thm:application SNC}
  Consider a finite connected simplicial complex $X$ and a K-homology class $\eta\in K_*(X)$. Assume that for each $\epsilon>0$
  there exists an $\epsilon$-flat Fredholm bundle $(E,F_E)$ over $X$, with arbitrary underlying unital C*-algebra $B_\epsilon$,
  such that $\langle\eta,\ind F_E\rangle\neq 0$. Then the image of $\eta$ under the analytic assembly map $\mu_X\colon K_*(X)\to
  K_*(C^*\pi_1(X;v_0))$ is nonzero.
\end{thm}

\begin{rem}
  The special case of Theorem~\ref{thm:application SNC} where the bundles $E$ are all finitely generated
  projective has been proved and used by Hanke and Schick
  \cite{hanke-schick-enlargeability-and-index-theory-I,hanke-schick-enlargeability-and-index-theory-infinite,
  hanke-schick-novikov-low-degree-cohomology,hanke-positive-scalar-curvature}. Note that in this case $F_E$ does not carry any
  information.
\end{rem}

Theorem~\ref{thm:application SNC} follows quite directly from Proposition~\ref{prop:calculation of asymptotic Kronecker
pairings} if the C*-algebras $B_\epsilon$ appearing in the statement of theorem are all equal to a single unital C*-algebra $B$.
Therefore, the missing step towards the proof of Theorem~\ref{thm:application SNC} is provided by the following lemma.

\begin{lem}\label{lem:pushforward of almost flat Fredholm bundles}
  Let $f\colon A\to B$ be a (not necessarily unital) *-ho\-mo\-mor\-phism between unital C*-algebras. Let $(E,F_E)$ be an
  $\epsilon$-flat Fredholm bundle over $X$, with underlying unital C*-algebra $A$. Then there exists an $\epsilon$-flat Fredholm
  bundle $(f_*E,f_*F_E)$ over $X$, with underlying unital C*-algebra $B$, such that
  \[
    \ind(f_*F_E)=(\id_{C(X)}\otimes f)_*\ind F_E\in K_0(C(X)\otimes B).
  \]
\end{lem}

Before we prove the lemma, we show how it can be used to derive Theorem~\ref{thm:application SNC}.

\begin{proof}[Proof of Theorem~\ref{thm:application SNC}]
  By the assumptions, there exists a sequence $(E_n,F_n)_{n\in\N}$ of $\epsilon_n$-flat Fredholm bundles over $X$, with
  underlying C*-algebras $B_n$, such that $\langle\eta,\ind F_n\rangle\neq 0\in K_*(B_n)$ for all $n\in\N$. Let $\iota_n\colon
  B_n\to\prod_{n\in\N}B_n=B$ be the inclusions. Then $\ind((\iota_n)_*F_n)=(\id_{C(X)}\otimes\iota_n)_*\ind F_n$ by
  Lemma~\ref{lem:pushforward of almost flat Fredholm bundles}, so that
  \[
    \langle\eta,\ind((\iota_n)_*F_n)\rangle=\langle\eta,(\id_{C(X)}\otimes\iota_n)_*\ind F_n\rangle=(\iota_n)_*\langle\eta,\ind
    F_n\rangle\in K_*(B).
  \]
  These elements of $K_*(B)$ are nonzero because each $\iota_n\colon B_n\to B$ is split injective (a splitting is given by the
  projection $B\to B_n$), so that $(\iota_n)_*\colon K_*(B_n)\to K_*(B)$ is injective. Therefore, we may replace the C*-algebras
  $B_n$ by the single C*-algebra $B$. Then $(E_n,F_n)_{n\in\N}$ is an asymptotically
  flat Fredholm bundle, and $\Psi((\langle\eta,\ind F_n\rangle)_{n\in\N})\neq 0$, so that indeed $\mu_X(\eta)\neq 0$ by
  Proposition~\ref{prop:calculation of asymptotic Kronecker pairings}.
\end{proof}

\begin{proof}[Proof of Lemma~\ref{lem:pushforward of almost flat Fredholm bundles}]
  Let $W$ be the typical fiber of $E$. We may consider the Hilbert $B$-module $f_*W=W\otimes_fB$
  \cite[Section~1.2]{jensen-thomsen-elements-kk-theory}. As a set, we put $f_*E=\sqcup_{x\in X}f_*E_x$, and consider the local
  trivializations
  \[
    f_*\Phi_v\colon S_v\times f_*W\to f_*E|_{S_v}
  \]
  which are determined by $f_*\Phi_v(x,\xi\otimes b)=\Phi_v(x,\xi)\otimes b\in f_*(E_x)$ for all $x\in S_v$, $\xi\in W$, and
  $b\in B$. Similarly, we define $f_*F_E\colon f_*E\to f_*E$ by
  \[
    f_*F_E|_{E_x\otimes_fB}=F_E|_{E_x}\otimes_f\id_B\colon E_x\otimes_fB\to E_x\otimes_fB.
  \]
  It is straightforward to check that $(f_*E,f_*F_E)$ is an $\epsilon$-flat Hilbert $B$-module bundle.

  Furthermore, one can show that the Kasparov $C(X;B)$-module
  \[
    (\Gamma(f_*E),\id,(f_*F_E)_*)
  \]
  is unitarily equivalent to $(\Gamma(E)\otimes_{\id\otimes f}C(X;B),\id,(F_E)_*\otimes\id)$, so that indeed
  \begin{align*}
    \ind f_*F_E&=\ind[\Gamma(E)\otimes_{\id\otimes f}C(X;B),\id,(F_E)_*\otimes\id]\\
               &=(\id\otimes f)_*\ind[\Gamma(E),\id,(F_E)_*]=(\id\otimes f)_*\ind F_E
  \end{align*}
  as claimed.
\end{proof}

\begin{rem}
  In \cite[Theorem~3.9]{hanke-positive-scalar-curvature}, Hanke formulated and proved Theorem~\ref{thm:application SNC} in the
  case of finite dimensional Hilbert module bundles. However, it was observed by Mario Listing that his proof needed the
  assumption that the transition functions of the Hilbert module bundles are Lipschitz with small Lipschitz constant.
  However, the proof of Theorem~\ref{thm:application SNC} does not need this assumption, so this clarifies that Theorem~3.9 of
  \cite{hanke-positive-scalar-curvature} is indeed correct under the assumptions stated there.

  Note that a simpler way of clarifying this point is to use Remark~\ref{rem:lipschitz-transition-functions-or-not} which shows
  that one can get transition functions with small Lipschitz constants when provided only with an almost flat bundle in the
  sense of this paper or in the sense of \cite{hanke-positive-scalar-curvature}.
\end{rem}

\section{Dadarlat's index theorem}

In this section, we will show how to use the methods of the foregoing section in order to generalize a theorem of
Dadarlat \cite[Theorem~3.2]{dadarlat-group-quasi-representations-index-theory} We consider the involutive Banach algebra
$\ell^1(G)$ which is the completion of the complex group algebra $\C G$ with respect to the 1-norm $\|\sum_{g\in G}\lambda_g\cdot
g\|=\sum_{g\in G}|\lambda_g|$.

This algebra has a property which makes it a very natural object to work with when studying almost flat bundles. Namely, if
$f\colon G\to V$ is any bounded map into a Banach space $V$, there is a unique extension to a bounded linear operator $\hat
f\colon\ell^1(G)\to V$. In addition, we will consider $\hat f^{(k)}\colon M_k(\ell^1(G))\to M_k(V)$ which is defined by applying
$\hat f$ on every matrix entry.

In particular, suppose that $\rho\colon\Fr(L)\to\mathcal L_B(W)$ is an $\epsilon$-representation of $G$ with respect to some
finite presentation $G=\langle L\mid R\rangle$. Choose a set-theoretic section $s\colon G\to\Fr(L)$ of the projection map
$\Fr(L)\to G$. Since $\rho\colon\Fr(L)\to\mathcal L_B(W)$ is an almost representation, in particular its image is contained in
the set of unitary operators on $W$. Thus, $\|\rho(w)\|\leq 1$ for all $w\in\Fr(L)$, so that $\rho\circ s\colon G\to\mathcal
L_B(W)$ is bounded. Thus, the construction above yields an extension $\hat\rho\colon\ell^1(G)\to\mathcal L_B(W)$ of $\rho\circ
s$, and maps $\hat\rho^{(k)}\colon M_k(\ell^1(G))\to M_k(\mathcal L_B(W))$. We observe the following fact:

\begin{prop}\label{prop:pushforward of projections over l1 along almost representations}
  Let $G=\langle L\mid R\rangle$ be a finitely presented group and let $s\colon G\to\Fr(L)$ be a set-theoretic section of the
  projection map $\pi\colon\Fr(L)\to G$. Let $p\in M_k(\ell^1(G))$ be a projection, and fix $\delta>0$. Then there exists a
  number $\epsilon=\epsilon(L,R,s,p,\delta)>0$ such that for every $\epsilon$-representation $\rho\colon\Fr(L)\to\mathcal
  L_B(W)$ the element $\hat\rho^{(k)}(p)\in M_k(\mathcal L_B(W))$ satisfies $\|\hat\rho^{(k)}(p)^2-\hat\rho^{(k)}(p)\|<\delta$
  and $\|\hat\rho^{(k)}(p)^*-\hat\rho^{(k)}(p)\|<\delta$.\qed
\end{prop}

In particular, if $\delta>0$ is small enough then Proposition~\ref{prop:pushforward of projections over l1 along almost
representations} implies that the self-adjoint matrix $\tilde p=\frac 12(\hat\rho^{(k)}(p)+\hat\rho^{(k)}(p)^*)$ satisfies
$\|\tilde p^2-\tilde p\|<\frac 14$. Consider the function $\psi\colon\R-\{\frac 12\}\to\R$ which is constantly equal to $0$ on
$\R_{<1/2}$ and constantly equal to $1$ on $\R_{>1/2}$. Then the element $\psi(\tilde p)\in M_k(\mathcal L_B(W))$, defined by
continuous functional calculus, is a well-defined projection. We put
\[
  \rho_\#(p)=\psi(\tilde p)
\]
in this case.

\begin{lem}\label{lem:pushforward of a projection along an almost Fredholm representation}
  Let $(W,\rho,F)$ be an $\epsilon$-Fredholm representation where $\epsilon>0$ is so small that $\rho_\#(p)$ as above is
  defined. Then $[\rho_\#(p),F\oplus\cdots\oplus F]\in M_k(\mathcal K_B(W))$.
\end{lem}
\begin{proof}
  Since $F^*-F$ is compact by assumption, the set of all matrices which commute with $F\oplus\cdots\oplus F$ up to $M_k(\mathcal
  K_B(W))$ is a C*-subalgebra of $M_k(\mathcal L_B(W))$, and in particular this set is closed under continuous functional
  calculus. It is thus enough to prove that $[\hat\rho^{(k)}(p),F\oplus\cdots\oplus F]\in M_k(\mathcal K_B(W))$. Equivalently,
  one has to prove that all entries of $\hat\rho^{(k)}(p)$ commute with $F$ up to $\mathcal K_B(W)$. Again, by the definition of
  $\hat\rho^{(k)}$ and by an approximation argument it suffices to prove that $[\hat\rho(g),F]\in\mathcal K_B(W)$ for all $g\in
  G$. However, $\hat\rho(g)=\rho(s(g))$ commutes with $F$ up to compact operators by definition of an almost Fredholm
  representation.
\end{proof}

In particular, consider the Kasparov $B$-module $W^k=W\oplus\cdots\oplus W$. Then Lemma~\ref{lem:pushforward of a projection
along an almost Fredholm representation} shows that $(W^k,\rho_\#(p),F\oplus\cdots\oplus F)$ defines a Kasparov $B$-module and
hence a class in $KK(B)$. We will prove a generalization of
\cite[Theorem~3.2]{dadarlat-group-quasi-representations-index-theory} which states that this construction relates to the pairing
of a K-homology class with an almost flat Fredholm bundle.

Lafforgue \cite{lafforgue-banach-kk-theory,lafforgue-k-theorie-bivariante} introduced the so-called $\ell^1$-assembly map
\[
  \mu_X^{\ell^1}\colon K_0(X)\to K_0(\ell^1(G))
\]
which has the property that the inclusion $i_{\ell^1}\colon\ell^1(G)\to C^*G$ satisfies
\[
  \mu_X=(i_{\ell^1})_*\circ\mu_X^{\ell^1}\colon K_0(X)\to K_0(C^*G).
\]
For a proof of this equation, one may, for instance, replace $C_r^*(G,B)$ by $C^*G$ in
\cite[Proposition~1.7.6]{lafforgue-k-theorie-bivariante}. Now we can state and prove the main theorem of this section.

\begin{thm}\label{thm:generalization of Dadarlat's index theorem}
  Let $\eta\in K_0(X)$ be a K-homology class of a finite connected simplicial complex $X$ with finitely presented fundamental
  group $G=\pi_1(X;v_0)=\langle L\mid R\rangle$. Choose simplicial loops $\Gamma_g$ representing the generators $g\in L$. Let
  $p,q\in M_k(\ell^1(G))$ be projections such that $\mu_X^{\ell^1}(\eta)=[p]-[q]\in K_0(\ell^1(G))$. Then there exists a number
  $\epsilon>0$ such that the following holds:

  Let $(E,F_E)$ be an $\epsilon$-flat Fredholm bundle over $X$, with arbitrary underlying unital C*-algebra $B$, and let
  $(W,\rho,F)$ be the associated almost Fredholm representation. Then $\rho_\#(p)$ and $\rho_\#(q)$ are defined, and
  \[
    \langle\eta,\ind F_E\rangle=\ind\left([W^k,\rho_\#(p),F\oplus\cdots\oplus F]-[W^k,\rho_\#(q),F\oplus\cdots\oplus F]\right)
  \]
  in $K_0(B)$, where $\ind\colon KK(B)\to K_0(B)$ is the index isomorphism.
\end{thm}
\begin{proof}
  If $(W,\rho,F)$ is an $\epsilon$-representation of $G$ over a C*-algebra $B$ and $p\in M_k(\ell^1(G))$ is a projection, then
  we abbreviate
  \[
    (W,\rho,F)_\#(p)=\ind[W^k,\rho_\#(p),F\oplus\cdots\oplus F]\in K_0(B),
  \]
  provided that $\epsilon$ is small enough such that the right hand side is defined. The proof of the theorem proceeds by
  contradiction. Thus, we assume that there is an asymptotically flat Fredholm bundle $(E_n,F_n)_{n\in\N}$ over $X$ with
  associated asymptotic Fredholm representation $(W_n,\rho_n,\hat F_n)_{n\in\N}$ such that
  \begin{equation}
    \langle\eta,\ind F_n\rangle\neq(W_n,\rho_n,\hat F_n)_\#(p)-(W_n,\rho_n,\hat F_n)_\#(q)\in K_0(B)
    \label{eq:contradicting equation for Dadarlat's theorem}
  \end{equation}
  for all $n\in\N$, where each $E_n$ is a Hilbert $B$-module bundle for a unital C*-algebra $B$.\footnote{A priori, each $E_n$
    is a Hilbert $B_n$-module bundle where $B_n$ depends on $n$. However, as in the proof of Theorem~\ref{thm:application SNC}
  we may replace each $B_n$ by $B=\prod_{n\in\N}B_n$.} By Proposition~\ref{prop:calculation of asymptotic Kronecker pairings} we
  have
  \[
    \Psi\left[\left(\langle\eta,\ind F_n\rangle\right)_{n\in\N}\right]=\asind\left((W_n,\rho_n,\hat
    F_n)_{n\in\N}\right)\bullet\Sigma\Phi(\mu_X(\eta)).
  \]
  Of course, $\mu_X(\eta)=(i_\ell^1)_*\mu_X^{\ell^1}(\eta)=(i_{\ell^1})_*[p]-(i_{\ell^1})_*[q]$. Therefore we get a
  contradiction to~\eqref{eq:contradicting equation for Dadarlat's theorem} if we can prove that
  \[
    \asind\left((W_n,\rho_n,\hat F_n)_{n\in\N}\right)\bullet\Sigma\Phi\left((i_{\ell^1})_*[p]-(i_{\ell^1})_*[q]\right)
  \]
  and
  \[
    \Psi\left(\left((W_n,\rho_n,\hat F_n)_\#(p)-(W_n,\rho_n,\hat F_n)_\#(q)\right)_{n\in\N}\right)
  \]
  define the same element of $D(\Sigma,B)$. We will actually prove that
  \begin{equation}
    \asind\left((W_n,\rho_n,\hat
      F_n)_{n\in\N}\right)\bullet\Sigma\Phi\left((i_{\ell^1})_*[p]\right)=\Psi\left(\left((W_n,\rho_n,\hat
    F_n)_\#(p)\right)_{n\in\N}\right)\label{eq:first reformulation of Dadarlat's theorem}
  \end{equation}
  for every projection $p\in M_k(\ell^1(G))$.

  Let us first calculate the right hand side of~\eqref{eq:first reformulation of the main theorem}.
  In order to do this, we need to analyze the classes $(W_n,\rho_n,\hat
  F_n)_\#(p)\in K_0(B)$. As in the definition of the asymptotic index, we choose even isometric isomorphisms $U_n\colon
  W_n\oplus\mathcal H_B\to\mathcal H_B$ and $V\colon\mathcal H_B'\to\mathcal H_B$. We consider $\hat F_n'$, $\rho_n'$,
  $\tilde\rho_n$, and $T'$ as in the definition of the asymptotic index of an asymptotic Fredholm representation.

  We define a *-ho\-mo\-mor\-phism $f\colon\mathcal L_B(W)\to\mathcal L_B(\mathcal H_B)$ by
  \[
    f(F)=VU(\hat F_n')((U_n(F\oplus 0)U_n^*\oplus 0)\oplus 0)U(\hat F_n')^*V^*.
  \]
  Then by definition we have $\tilde\rho_n(w)=f(\rho_n(w))$ for all $w\in\Fr(L)$, and therefore
  \[
    \tilde\rho_{n\#}(p)=\id_{M_k}\otimes f(\rho_{n\#}(p)).
  \]
  Let us write $\rho_{n\#}(p)=(p_{jl})_{j,l}\in M_k(\mathcal
  L_B(W))$. We abbreviate $p_{jl}'=U_n(p_{jl}\oplus 0)U_n^*\in\mathcal L_B(W)$, $\hat F_n''=(\hat F_n'\oplus(-\hat
  F_n'))\oplus((-\hat F_n')\oplus\hat F_n')\in\mathcal L_B(\mathcal H_B')$, and $\tilde F_n=V\mathcal U(\hat F_n')\hat
  F_n''\mathcal U(\hat F_n')^*V^*$. Then
  \[
    [(W_n)^k,\rho_{n\#}(p),\hat F_n\oplus\cdots\oplus\hat F_n]
    =[(\mathcal H_B)^k,\tilde\rho_{n\#}(p),\tilde F_n\oplus\cdots\oplus\tilde F_n]
  \]
  in $KK(B)$. Furthermore, it can be shown straightforwardly that $T'\oplus\cdots\oplus T'$ is a compact perturbation
  of $((\mathcal H_B)^k,\tilde\rho_{n\#}(p),\tilde F_n\oplus\cdots\oplus\tilde F_n)$. In summary,
  \begin{align*}
    (W_n,\rho_n,\hat F_n)_\#(p)&=\ind[(W_n)^k,\rho_{n\#}(p),\hat F_n\oplus\cdots\oplus\hat F_n]\\
                               &=\ind'[(\mathcal H_B)^k,\tilde\rho_{n\#}(p),T'\oplus\cdots\oplus T'],
  \end{align*}
  where we used that $\ind=\ind'$ by Theorem~\ref{thm:ind=ind'}. This index can be calculated in way which is similar to the
  calculation of $\ind'[\hat E']$ in the proof of Theorem~\ref{thm:alternative calculation of the index of an almost flat
  Fredholm bundle}. This calculation yields
  \[
    (W_n,\rho_n,\hat F_n)_\#(p)=\rho[\tilde\rho_{n\#}(p)].
  \]
  We refer to \cite[Theorem~5.2.4]{hunger-asymptotically-flat-fredholm-bundles} for details.

  Let $\sigma\in E(Q,\mathcal K_B(H_B))$ be the E-theory element associated to the sequence
  \[
    \begin{tikzpicture}
      \matrix (m) [matrix of math nodes, row sep=3em, column sep=3em, commutative diagrams/every cell]
      {
        0 & \mathcal K_B(H_B) & Q & \mathcal L_B(H_B) & 0. \\
      };
      \path[-stealth, commutative diagrams/.cd, every label] (m-1-1) edge (m-1-2) (m-1-2) edge (m-1-3) (m-1-3) edge (m-1-4)
        (m-1-4) edge (m-1-5);
    \end{tikzpicture}
  \]
  by Proposition~\ref{prop:E-theory element associated to a split short exact sequence}. As in the proof of
  Theorem~\ref{thm:main theorem} it follows from Lemma~\ref{lem:homomorphism associated to a split exact sequence of abelian
  groups} and from the naturality of $\Psi$ that
  \[
    \Psi\left(\left((W_n,\rho_n,\hat
    F_n)_\#(p)\right)_{n\in\N}\right)=\Psi\left(\left(\rho[\tilde\rho_{n\#}(p)]\right)_{n\in\N}\right)
    =\sigma\bullet\Psi\left(\left([\tilde\rho_{n\#}(p)]\right)_{n\in\N}\right)
  \]
  in $D(\Sigma,\mathcal K_B(H_B))$.

  On the other hand, note that in the left hand side of~\eqref{eq:first reformulation of Dadarlat's theorem} we have
  \[
    \asind((W_n,\rho_n,\hat F_n)_{n\in\N})=\sigma\bullet[\Sigma^2\rho\otimes\id_{\mathcal K}],
  \]
  where $\rho$ is as in Lemma~\ref{lem:definition of rho for the asymptotic index}. In particular, \eqref{eq:first reformulation
  of Dadarlat's theorem}~follows if we can prove that
  \begin{equation}
    \Psi\left(\left([\tilde\rho_{n\#}(p)]\right)_{n\in\N}\right)=[\Sigma^2\rho\otimes\id_{\mathcal
    K}]\bullet\Sigma\Phi\left((i_{\ell^1})_*[p]\right)\in D(\Sigma,Q).
    \label{eq:second reformulation of Dadarlat's theorem}
  \end{equation}
  By the definition of $\Psi$ we have
  $\Psi(([\tilde\rho_{n\#}(p)])_{n\in\N})=[\Sigma^2f_{(\tilde\rho_{n\#}(p))}\otimes\id_{\mathcal K}]$ where
  $f_{(\tilde\rho_{n\#}(p))}\colon\C\to\mathcal A_\delta(Q\otimes\mathcal K)$ is the unique *-ho\-mo\-mor\-phism such that
  $f_{(\tilde\rho_{n\#}(p))}(1)=[n\mapsto\tilde\rho_{n\#}(p)]$. Of course, we have $[n\mapsto\tilde\rho_{n\#}(p)]=
  [n\mapsto\hat{\tilde\rho}_n^{(k)}(p)]\in\mathcal A_\delta(Q\otimes\mathcal K)$ since
  $\lim_{n\to\infty}\|\hat{\tilde\rho}_n^{(k)}(p)-\tilde\rho_{n\#}(p)\|=0$. On the other hand,
  $\Phi((i_{\ell^1})_*[p])=\kappa(\Sigma f_{(i_{\ell^1})_*(p)}\otimes\id_{\mathcal K})$ where $f_{(i_{\ell^1})_*(p)}\colon\C\to
  C^*G\otimes\mathcal K$ is such that $f_{(i_{\ell^1})_*(p)}(1)=i_{\ell^1}\otimes\id_{\mathcal K}(p)$. Thus,
  \eqref{eq:postcomposition with an asymptotic homomorphism}~implies that the right hand side of \eqref{eq:second reformulation
  of Dadarlat's theorem} is given by $[\Sigma^2h_p\otimes\id_{\mathcal K}]$ where $h_p=(\rho\otimes\id_{\mathcal K})\circ
  f_{(i_{\ell^1})_*(p)}\colon\C\to\mathcal A_\delta(Q\otimes\mathcal K)$ is determined by $h_p(1)=\rho\otimes\id_{\mathcal
  K}(i_{\ell^1}\otimes\id_{\mathcal K}(p))$. In particular,
  \[
    h_p(1)=[n\mapsto\hat{\tilde\rho}_n^{(k)}(p)]=f_{(\tilde\rho_{n\#}(p))}(1)
  \]
  by the definition of $\rho$, so that actually $h_p=f_{(\tilde\rho_{n\#}(p))}$. This proves \eqref{eq:second reformulation of
  Dadarlat's theorem} and hence completes the proof of the theorem.
\end{proof}

\begin{rem}
  Dadarlat's theorem \cite[Theorem 3.2]{dadarlat-group-quasi-representations-index-theory} is the specialization of
  Theorem~\ref{thm:generalization of Dadarlat's index theorem} to the case where the bundles are not almost flat Fredholm
  bundles but finite-dimensional almost flat bundles (in which case the Fredholm operator is neither necessary nor carries any
  important information). The proof of Theorem~\ref{thm:generalization of Dadarlat's index theorem} in this special case can be
  carried out replacing the use of Proposition~\ref{prop:calculation of asymptotic Kronecker pairings} with the arguments of the
  proof of \cite[Theorem 3.9]{hanke-positive-scalar-curvature}. Thus, we have also given a fairly simple proof of Dadarlat's
  theorem which is quite different from Dadarlat's original proof.
\end{rem}

\printbibliography

\end{document}